\documentclass[leqno,a4paper,twoside,11pt]{article}
\usepackage[nointlimits,nosumlimits]{amsmath}
\usepackage{amsfonts,amssymb,amsthm,ifthen}
\usepackage[utf8x]{inputenc}

\newtheorem{Thm}{Theorem}[section]
\newtheorem{Prp}[Thm]{Proposition}
\newtheorem{Lem}[Thm]{Lemma}

\newtheorem{Def}[Thm]{Definition}
\newtheorem{Rem}[Thm]{Remark}

\newcommand{\id}{\mathrm{id}}
\newcommand{\dvol}{\mathrm{dvol}}

\newcommand{\Hom}{\mathrm{Hom}}

\newcommand{\setsep}{\;\big|\;}
\newcommand{\dbar}{\overline{\partial}}

\newcommand{\fE}{\mathfrak{E}}
\newcommand{\fF}{\mathfrak{F}}
\newcommand{\fG}{\mathfrak{G}}
\newcommand{\mA}{\mathcal A}
\newcommand{\mB}{\mathcal B}
\newcommand{\mD}{\mathcal D}
\newcommand{\mE}{\mathcal E}
\newcommand{\mF}{\mathcal F}
\newcommand{\mG}{\mathcal G}
\newcommand{\mH}{\mathcal H}
\newcommand{\mJ}{\mathcal J}

\newcommand{\mM}{\mathcal M}

\newcommand{\mP}{\mathcal P}
\newcommand{\mQ}{\mathcal Q}

\newcommand{\mT}{\mathcal T}
\newcommand{\mU}{\mathcal U}
\newcommand{\mW}{\mathcal W}
\newcommand{\bC}{\mathbb C}
\newcommand{\bN}{\mathbb N}
\newcommand{\bR}{\mathbb R}
\newcommand{\bZ}{\mathbb Z}
\newcommand{\fg}{\mathfrak g}
\newcommand{\os}{\overline{s}}

\newcommand{\tvarphi}{\tilde{\varphi}}

\newcommand{\tD}{\tilde{D}}
\newcommand{\tJ}{\tilde{J}}
\newcommand{\tL}{\tilde{L}}

\newcommand{\norm}[1]{\left|\left| #1 \right|\right|}

\newcommand{\scal}[3][]{\ifthenelse{\equal{#1}{}}{
  \left\langle #2,\,#3 \right\rangle
}{\ifthenelse{\equal{#1}{(}}{
  \left( #2,\,#3 \right)
}{\ifthenelse{\equal{#1}{[}}{
  \left[ #2,\,#3 \right]
}{
  #1\left( #2,\,#3 \right)
}}}}

\renewcommand{\title}[1]{\vbox{\center\LARGE{\textsc{#1}}}\vspace{5mm}}
\renewcommand{\author}[1]{\vbox{\center\large{\textsc{#1}}}\vspace{5mm}}
\newcommand{\address}[1]{\vbox{\center\em#1}}
\newcommand{\email}[1]{\vbox{\center\tt#1}\vspace{5mm}}

\begin{document}

\title{Transversality for Holomorphic Supercurves}

\author{Josua Groeger$^1$}

\address{Humboldt-Universit\"at zu Berlin, Institut f\"ur Mathematik,\\
  Rudower Chaussee 25, 12489 Berlin, Germany }

\email{$^1$groegerj@mathematik.hu-berlin.de}


\begin{abstract}
\noindent
We study holomorphic supercurves, which are motivated by supergeometry as a natural generalisation
of holomorphic curves.
We prove that, upon perturbing the defining equations by making them depend on a connection,
the corresponding linearised operator is generically surjective. By this transversality result,
we show that the resulting moduli spaces are oriented finite dimensional smooth manifolds.
Finally, we examine how they depend on the choice of generic data.
\end{abstract}


\section{Introduction}

Starting with the seminal work of Gromov (\cite{Gro85}), holomorphic curves have become
a powerful tool in the study of symplectic geometry.
In general, solution sets of nonlinear elliptic differential equations often lead
to interesting algebraic invariants. In the case of holomorphic curves, these are known
as Gromov-Witten invariants.
Here, one typically encounters two problems: Firstly, the moduli spaces in question are
wanted to have a nice structure (which is a transversality problem) and, secondly,
they are usually not compact but need to be compactified.
Examples for the occurrence and solution of both issues include the aforementioned
Gromov-Witten invariants (cf. \cite{MS04}), invariants of Hamiltonian group actions
as introduced in \cite{CGMS02} and symplectic field theory (cf. \cite{EGH00} for
a general introduction and \cite{Dra04} for transversality).

Holomorphic supercurves were introduced in \cite{Gro11} as a natural generalisation of holomorphic curves motivated by supergeometry. Indeed, for a careful choice of
geometric data, the defining equation $\dbar_J\Phi=0$ continues to make sense for
morphisms of supermanifolds. It can be translated into differential equations
for an ordinary map and sections of vector bundles, which we shall recall in the
next section.

The aim of the present article is to solve the transversality problem in this context,
while a companion article (\cite{Gro11c}) is concerned with compactness.
We define $(A,J)$-holomorphic supercurves by perturbing the defining equations
by making them depend on a connection, an approach that will prove successful later.
In fact, our definition is such that the new objects under consideration
satisfy a slightly weaker condition than (genuine) holomorphic supercurves
in the non-Kähler case. This change of perspective is ultimately justified
by our transversality and compactness results, which raise hope
to be able to construct new invariants or at least to find new expressions for existing
ones in subsequent work.

This article may be read independent of \cite{Gro11} and \cite{Gro11c} and is organised as follows.
In Sec. \ref{secAJ}, we define $(A,J)$-holomorphic supercurves as advertised above.
In Sec. \ref{secProperties}, we show local properties of single holomorphic supercurves
such as elliptic regularity.
In Sec. \ref{secModuliSpaces}, the main part of this article, we introduce moduli spaces of holomorphic supercurves
and prove that they are, generically, smooth oriented manifolds.
We also specify a sufficient condition for existence.
Finally, we show that the moduli spaces are, for different choices of generic data,
oriented cobordant (usually without being compact, however).

\section{$(A,J)$-Holomorphic Supercurves}
\label{secAJ}

To solve the transversality problem for holomorphic supercurves, it is necessary to
perturb the defining equations by making them depend on another parameter, which we
choose to be a connection. Let us first briefly recall the mathematical background.
Consult \cite{MS04} as well as the references therein for details on holomorphic curves
and \cite{Gro11} for a motivation of holomorphic supercurves in terms of supergeometry.

Let $\Sigma$ be a connected and closed Riemann surface with a fixed complex structure $j$
and $(X,\omega)$ be a compact symplectic manifold of dimension $2n$.
We consider $\omega$-tame (or $\omega$-compatible) almost complex structures $J$ on $X$.
Every such structure $J$ determines a Riemann metric $g_J$.
A $J$-holomorphic curve is a smooth map $\varphi:\Sigma\rightarrow X$
such that $\dbar_J\varphi:=\frac{1}{2}(d\varphi+J\circ d\varphi\circ j)=0$ vanishes.
We denote the moduli space of simple (i.e. not multiply covered) $J$-holomorphic curves
representing the homology class $\beta\in H_2(X,\bZ)$ by
\begin{align}
\label{eqnModuliSpaceSimpleCurves}
\mM^*(\beta,\Sigma;J):=\{\varphi\in C^{\infty}(\Sigma,X)\setsep
\dbar_J\varphi=0\,,\;[\varphi]=\beta\,,\;\textrm{$\varphi$ is simple}\}
\end{align}
If $J$ is generic, $\mM^*(\beta,\Sigma;J)$ is an oriented smooth manifold of dimension
\begin{align}
\label{eqnDimensionClassicalModuliSpace}
\dim\mM^*(\beta,\Sigma;J)=n(2-2g)+2c_1(\beta)
\end{align}
where $g$ denotes the genus of $\Sigma$, and $c_1$ the first Chern class of $(TX,J)$.
This transversality result follows from the infinite dimensional implicit function theorem
after showing that the linearisation $D_{\varphi}$ of $\dbar_J$ at $\varphi$,
considered as a map between appropriate Banach spaces, is a real linear (smooth)
Cauchy-Riemann operator with Fredholm index (\ref{eqnDimensionClassicalModuliSpace}), which
is moreover surjective for generic $J$ as shown by a Sard-Smale argument.
For general smooth maps $\varphi:\Sigma\rightarrow X$, $D_{\varphi}$ can be defined as an operator
$D_{\varphi}:\Omega^0(\Sigma,\varphi^*TX)\rightarrow\Omega^{0,1}(\Sigma,\varphi^*TX)$
by prescribing
\begin{align}
\label{eqnLinearisedDbar}
D_{\varphi}\xi:=\frac{1}{2}(\nabla\xi+J(\varphi)\nabla\xi\circ j)
-\frac{1}{2}J(\varphi)(\nabla_{\xi}J)(\varphi)\partial_J(\varphi)
\end{align}
which coincides with the linearisation of $\dbar_J$ at $\varphi$ if $\varphi$ is holomorphic,
where $\nabla:=\nabla^{g_J}$ denotes the Levi-Civita connection (and its pullback) of $g_J$.
If $J$ is $\omega$-compatible, this operator can be expressed as the sum
\begin{align}
\label{eqnLinearisedDbarCompatible}
D_{\varphi}\xi=(\nabla^J\xi)^{0,1}+\frac{1}{4}N_J(\xi,\partial_J\varphi)
\end{align}
of a complex linear Cauchy-Riemann operator and a zero order complex antilinear operator.
Here, $N_J$ is the Nijenhuis tensor and $\nabla^J$ and $(\nabla^J)^{0,1}$ denote
the $J$-complexification of $\nabla$ and the $(0,1)$-part of $\nabla^J$, respectively.

\begin{Def}[\cite{Gro11}]
\label{defHolomorphicSupercurveAdhoc}
Let $L\rightarrow\Sigma$ be a holomorphic line bundle.
Then a \emph{holomorphic supercurve} is a tuple $(\varphi,\psi_1,\psi_2,\xi)$,
consisting of a smooth map $\varphi:\Sigma\rightarrow X$ and sections
\begin{align*}
\psi_1,\psi_2\in\Gamma(\Sigma,L\otimes_J\varphi^*TX)\;&,\quad\xi\in\Gamma(\Sigma,\varphi^*T^{\bC}X)\\
N_J^{\bC}(\psi_{1\theta},\psi_{2\theta})=0\,,\quad\dbar_J\varphi=0\,&,\quad
D_{\varphi}^{\bC}\xi=0\,,\quad D_{\varphi}^{\bC}\psi_{1\theta}=0\,,\quad D_{\varphi}^{\bC}\psi_{2\theta}=0
\end{align*}
where, for $U\subseteq\Sigma$ sufficiently small, we fix a nonvanishing holomorphic section
$\theta\in\Gamma(U,L)$ and let $\psi_{j\theta}\in\Gamma(U,\varphi^*T^{1,0}X)$ be
such that $\psi_j=\theta\cdot\psi_{j\theta}$ holds for $j=1,2$.
\end{Def}

$D_{\varphi}^{\bC}$ and $N_J^{\bC}$ are the complex-linear extensions of $D_{\varphi}$ and $N_J$, respectively.
Upon identifying $\psi_{j\theta}$ with a (local) section of $\varphi^*TX$,
the last two conditions may be reformulated into
$D_{\varphi}\psi_{j\theta}=D_{\varphi}(J\psi_{j\theta})=0$.
If $J$ is $\omega$-compatible, this is by (\ref{eqnLinearisedDbarCompatible}) equivalent to
\begin{align}
\label{eqnSimplifiedHolomorphicSupercurve}
(\nabla^J\psi_{j\theta})^{0,1}=0\;,\qquad N_J(\psi_{j\theta},\partial_J\varphi)=0
\end{align}
For simplicity, we shall restrict ourselves to the case $\psi_2=0$ and $\xi=0$ in the following.
In the $\omega$-compatible case,
a holomorphic supercurve is then a pair $(\varphi,\psi)$, consisting of a holomorphic curve
$\varphi:\Sigma\rightarrow X$ and a section $\psi\in\Gamma(\Sigma,L\otimes_J\varphi^*TX)$
such that (\ref{eqnSimplifiedHolomorphicSupercurve}) holds with
$\psi_j$ replaced by $\psi$.
By the proof of Prp. \ref{prpUniversalModuliSpace}, it will become clear that for establishing
transversality, we need to perturb (\ref{eqnSimplifiedHolomorphicSupercurve}) by another parameter,
which we choose to be a connection $\nabla^A$ on $X$. We will moreover drop the condition
$N_J(\psi_{\theta},\partial_J\varphi)=0$. We define an operator
$\mD_{\varphi}^{A,J}:\Omega^0(L\otimes_J\varphi^*TX)\rightarrow\Omega^{0,1}(L\otimes_J\varphi^*TX)$
by prescribing
\begin{align}
\label{eqnDOperator}
\mD_{\varphi}^{A,J}\psi
:=\left(\nabla^{A,J}\psi_{\theta}\right)^{0,1}\cdot\theta+\psi_{\theta}\cdot(\dbar\theta)
\end{align}
where, for $U\subseteq\Sigma$ sufficiently small, we fix a nonvanishing section
$\theta\in\Gamma(U,L)$ and let $\psi_{j\theta}\in\Gamma(U,\varphi^*TX)$ be
such that $\psi_j=\theta\cdot\psi_{j\theta}$ holds. Here, $\dbar$ denotes the
usual Dolbeault operator on $L$.
By a straightforward calculation, this is well-defined and makes
$\mD^{A,J}_{\varphi}$ a complex linear Cauchy-Riemann operator.
We thus arrive at the main definition of this article.

\begin{Def}
\label{defAJHolomorphicSupercurve}
An \emph{$(A,J)$-holomorphic supercurve} is a pair $(\varphi,\psi)$, consisting of a smooth map
$\varphi\in C^{\infty}(\Sigma,X)$ and a section $\psi\in\Gamma(\Sigma,L\otimes_J\varphi^*TX)$,
for brevity denoted $(\varphi,\psi):(\Sigma,L)\rightarrow X$, such that
\begin{align*}
\dbar_J\varphi=0\;,\qquad\mD^{A,J}_{\varphi}\psi=0
\end{align*}
holds.
\end{Def}

It is clear that, if $(X,\omega,J)$ is Kähler,
any $(A^{g_J},J)$-holomorphic supercurve $(\varphi,\psi)$ may be identified with a holomorphic supercurve
$(\varphi,\psi,0,0)$ in the sense of Def. \ref{defHolomorphicSupercurveAdhoc},
where $\nabla^{A^{g_J}}=\nabla^{g_J}$ denotes the Levi-Civita connection of $g_J$.
However, this characterisation will not be needed in the following, and we shall simply take
Def. \ref{defAJHolomorphicSupercurve} as our starting point.
We close this section with some simple observations about existence in special cases.

\begin{Rem}
Let $(\varphi,\psi):(\Sigma,L)\rightarrow X$ be a holomorphic supercurve such that $\varphi$
is constant. In this case, the operator $\mD_{\varphi}^{A,J}$ reduces to the usual Dolbeault
operator $\dbar$ on $L\otimes_{J(\varphi(0))} T_{\varphi(0)}X$.
If, in addition, $\deg L=c_1(L)<0$ holds there are no nonzero global holomorphic sections on $L$
(cf. \cite{GH78}) and thus $\psi\equiv 0$ vanishes identically.
\end{Rem}

\begin{Rem}
Let $(X,\omega,J)$ be a Kähler manifold.
Let $\varphi:S^2\rightarrow X$ be a holomorphic sphere and assume that $\xi\in\Gamma(\varphi^*TX)$
satisfies $D_{\varphi}\xi=0$. Geometrically, this means that $\xi$ is tangent to $\varphi$.
If, moreover, the line bundle $L$ allows a global holomorphic section $\theta$,
then $(\varphi,\psi:=\theta\cdot\xi)$ is an $(A^{g_J},J)$-holomorphic supercurve.
If $L$ has no global holomorphic sections, we may still consider holomorphic sections
$\theta$ of $L|_{S^2\setminus\{0\}}$. Of course, $\theta$ has no extension
over $0$, but $\psi:=\theta\cdot\xi$ does, provided that $\xi$ vanishes in $0$ to
a sufficiently large order. In this case, we again obtain an $(A^{g_J},J)$-holomorphic supercurve.
Consult Chp. 6 of \cite{CM07} for such tangency conditions.
\end{Rem}

\begin{Rem}
Let $J$ be an $\omega$-compatible almost complex structure. Fix a spin structure on $\Sigma$
and let $S=S^+\oplus S^-$ denote the complex spinor bundle such that $(S^+)^2\cong T^*\Sigma$.
Let $\varphi\in C^{\infty}(\Sigma,X)$ be a $J$-holomorphic curve and $\zeta\in\Gamma(S^-)$
be a holomorphic (half-)spinor.
By Lem. 3.18 of \cite{Gro11}, the tuple $(\varphi,\psi^{\varphi,\zeta},0,0)$,
with $\psi^{\varphi,\zeta}\in\Gamma(S^+\otimes_J\varphi^*TX)$ as defined there,
is a holomorphic supercurve in the sense of Def. \ref{defHolomorphicSupercurveAdhoc}.
Therefore, $(\varphi,\psi^{\varphi,\zeta})$ is also an $(A^{g_J},J)$-holomorphic supercurve
in the sense of Def. \ref{defAJHolomorphicSupercurve}.
\end{Rem}

\section{Local Holomorphic Supercurves}
\label{secProperties}

In this section, we reveal local properties of single $(A,J)$-holomorphic supercurves
where, for the transversality arguments in Sec. \ref{secModuliSpaces}, we weaken the regularity
assumptions towards Sobolev spaces. For a standard treatment on the latter, consult \cite{Dob06}.
For convenience, we shall briefly recall some basic properties needed later on.
Sobolev spaces embed into spaces of differentiable functions:
Let $k>0$ and $p>n$, and assume that $U\subseteq\bR^n$ is open, bounded and has a Lipschitz
boundary.
Then there is a constant $C>0$, depending on $U$, $k$ and $p$ such that
\begin{align}
\label{eqnMorrey}
\norm{f}_{C^{k-1}(U)}\leq C\cdot\norm{f}_{U,k,p}
\end{align}
The inclusion $W^{k,p}(U)\subseteq C^{k-1}(U)$ is compact.
If $kp>n$ there is, moreover,
a constant $c>0$, depending on $U$, $k$ and $p$ such that
\begin{align}
\label{eqnProductSobolev}
\norm{fg}_{U,j,p}\leq c\norm{f}_{U,j,p}\norm{g}_{U,k,p}
\end{align}
for $j\leq k$. In particular, products of $W^{k,p}$-functions are again of this regularity class.
With the same hypotheses, we obtain that the composition $f\circ g$ of a $W^{k,p}$-map $g$ with a $C^k$-map
$f$ is again of class $W^{k,p}$.
Maps between manifolds and sections of vector bundles are
defined to be of regularity class $W^{k,p}$ if the same is true in local coordinates
(and trivialisations). Note that for this to be well-defined, i.e. independent of the choice of the
coordinates, is is necessary to assume $kp>n$ in order to ensure (\ref{eqnProductSobolev})
and the subsequent remark on compositions, both applied to coordinate changes and transition maps.
In our case of $n=2$, this amounts to requiring $p>2$.
Consult App. B of \cite{Weh04} for an alternative treatment on such Sobolev spaces.

\begin{Def}
\label{defWeakHolomorphicSupercurve}
Let $A\in\mA_{C^{l+1}}(GL(X))$ be a connection and $J\in C^{l+1}(X,\mathrm{End}(TX))$
be an almost complex structure on $X$, both of class $C^{l+1}$ with $l\geq 1$, and let $p>2$.
Then an \emph{$(A,J)$-holomorphic supercurve (of regularity class $W^{1,p}$)} is a pair
$(\varphi,\psi)$, consisting of a map $\varphi\in W^{1,p}(\Sigma,X)$ and a section
$\psi\in W^{1,p}(\Sigma,L\otimes_J\varphi^*TX)$,
such that $\dbar_J\varphi=0$ and $\mD_{\varphi}^{A,J}\psi=0$.
\end{Def}

For the subsequent local analysis, consider the building block
$\left(\nabla^{A,J}\right)^{0,1}$ of the operator $\mD_{\varphi}^{A,J}$.
Let $A\in\mA(GL(X))\subseteq\Omega^1(GL(X),gl(2n))$ be the connection on the principal bundle $GL(X)$ of frames
which corresponds to $\nabla^A$ (cf. \cite{Bau09} and \cite{KN96}).
Let $\varphi\in C^{\infty}(\Sigma,X)$ be a smooth map.
Let $V\subseteq X$ be a sufficiently small open subset such that $TX|_V\cong\bR^{2n}$ is trivial
and set $U:=\varphi^{-1}(V)\subseteq\Sigma$.
Let $\os:V\subseteq X\rightarrow GL(X)|_V$ be a smooth (nonvanishing) local section and
$s:=\os\circ\varphi:U\rightarrow\varphi^*GL(X)|_U$.
Then, identifying a vector field $\xi\in\Gamma(U,\varphi^*TX)$ with a map
$v\in C^{\infty}(U,\bR^{2n})$ via $\xi=[s,v]$, we yield the local formula
\begin{align*}
\nabla^A\xi=\left[s\,,\;dv[\cdot]+(A\circ ds[\cdot])\cdot v\right]
\end{align*}
If, moreover, $J\in\Gamma(\mathrm{End}(TX))$ is an almost complex structure on $X$, a short
calculation yields
\begin{align}
\label{eqnLocalComplexConnection}
\left(\nabla^{A,J}\xi\right)^{0,1}
=\scal[[]{s}{\dbar_Jv+\frac{1}{4}\left((C\circ d\varphi)+(J\circ\varphi)\cdot(C\circ d\varphi)\circ j\right)v}
\end{align}
where, on the right, $J$ is identified with a map $J\in C^{\infty}(V,GL(2n,\bR))$ and
\begin{align*}
C:=(A\circ d\os)-J\cdot (A\circ d\os)\cdot J-J\cdot dJ\in\Omega^1(V,\fg)
\end{align*}
Note that (\ref{eqnLocalComplexConnection}) continues to
hold accordingly for $A$ and $J$ of regularity class $C^{l+1}$ instead of $C^{\infty}$.
In the following, we denote by $z=s+it$ the standard coordinates on $\bC\cong\bR^2$.

\begin{Def}
\label{defLocalHolomorphicSupercurve}
Let $U\subseteq\bC$ be an open set and $p>2$. Then a pair of functions $(\varphi,\psi)\in W^{1,p}(U,\bR^{2n})$
is called a \emph{local holomorphic supercurve} if there are functions
$J\in C^{l+1}(V,\bR^{2n\times 2n})$ and $D\in C^l(V\times\bR^{2n},\bR^{2n\times 2n})$,
where $V\subseteq\bR^{2n}$ is an open set with $\varphi(U)\subseteq V$, $J^2=-\id$ holds
and $D$ is linear in the second component, such that,
abbreviating $\tilde{J}:=J\circ\varphi$ and $\tilde{D}:=(D(\varphi)\cdot\partial_s\varphi)$,
\begin{align}
\label{eqnLocalHolomorphicCurve}
\partial_s\varphi(z)+\tilde{J}(z)\cdot\partial_t\varphi(z)&=0\\
\label{eqnLocalHolomorphicSupercurve}
\partial_s\psi(z)+\tilde{J}(z)\cdot\partial_t\psi(z)+\tilde{D}(z)\cdot\psi(z)&=0
\end{align}
holds for every $z\in U$.
\end{Def}

It is clear that $\tilde{J}\in W^{1,p}(U,\bR^{2n\times 2n})$
and $\tilde{D}\in L^p(U,\bR^{2n\times 2n})$ hold.
$\tJ$ and $\tD$ are of higher regularity if $\varphi$ is.
The following result explains the terminology.

\begin{Lem}
\label{lemLocalHolomorphicSupercurve}
Let $(A,J)$ be of regularity $C^{l+1}$ and $(\varphi,\psi):(\Sigma,L)\rightarrow X$ be an
$(A,J)$-holomorphic supercurve of class $W^{1,p}$ with $p>2$. Let $V\subseteq X$ be a
coordinate chart of $X$ such that $L$ is trivial on $U:=\varphi^{-1}(V)$, which we then identify
with a subset of $\bC$ via holomorphic coordinates $z=s+it$. Then, upon restricting to $U$,
we may identify $(\varphi,\psi)$ with a local holomorphic supercurve.
\end{Lem}

\begin{proof}
The statement for $\varphi$ is clear. Now choose a holomorphic
section $\theta\in\Gamma(U,L)$ to (locally) identify $L\otimes_{\bC}\varphi^*T^{1,0}X$
with $\varphi^*TX$. Moreover, note that any section of $GL(X)$ determines a (local) trivialisation of $TX$,
which applies in particular to the local frame determined by a (smooth) coordinate chart of $X$.
Therefore, (\ref{eqnLocalComplexConnection}) yields $C$ (of regularity class $C^l$)
such that the defining equation $\mD_{\varphi}^{A,J}\psi=0$ for $\psi$ is, upon
applying to $\partial_s$, equivalent to (\ref{eqnLocalHolomorphicSupercurve})
with $D:=\frac{1}{2}(C+J\cdot C \circ J)$.
\end{proof}

By (\ref{eqnLocalHolomorphicSupercurve}), holomorphic supercurves satisfy a perturbed $\dbar_J$-equation.
As a first consequence, we may apply the Carleman similarity principle (cf. \cite{FHS95}), to be stated next.
Here, the $\varphi$-dependence of $\tJ$ and $\tD$ is not important.

\begin{Lem}[Carleman similarity principle]
\label{lemCarlemanSimilarity}
Let $0\in U\subseteq\bC$ be an open set, $p>2$,
$\tD\in L^p(U,\bR^{2n\times 2n})$ and $\tJ\in W^{1,p}(U,\bR^{2n\times 2n})$ be such that $\tJ^2=-\id$.
Let $\psi\in W^{1,p}(U,\bR^{2n})$ be a solution of (\ref{eqnLocalHolomorphicSupercurve}) such that $\psi(0)=0$.
Then there is a smaller open set $0\in U'\subseteq U$, a map $\Phi\in W^{1,p}(U',\Hom_{\bR}(\bC^n,\bR^{2n}))$
and a holomorphic map $\sigma:U'\rightarrow\bC^n$ such that $\Phi(z)$ is invertible and
\begin{align*}
\psi(z)=\Phi(z)\sigma(z)\;,\qquad\sigma(0)=0\;,\qquad\Phi(z)^{-1}\tJ(z)\Phi(z)=i
\end{align*}
holds for every $z\in U'$.
\end{Lem}

\begin{Lem}
\label{lemCROperatorZeroes}
Let $(A,J)$ be of regularity $C^{l+1}$ and $(\varphi,\psi):(\Sigma,L)\rightarrow X$
be an $(A,J)$-holomorphic supercurve of class $W^{1,p}$. Then either
$\psi\equiv 0$ or the set of zeros $\{z\in\Sigma\setsep\psi(z)=0\}$ is finite.
\end{Lem}

\begin{proof}
By Lem. \ref{lemLocalHolomorphicSupercurve} and Lem. \ref{lemCarlemanSimilarity}, we find around each
$z_0\in\Sigma$ an open set $U'(z_0)$ such that $(\varphi,\psi)$ is a local holomorphic supercurve
and, moreover, $\psi$ is similar to a holomorphic map. Since $\Sigma$ is compact, it is covered by finitely
many such sets $U'_i$.
Now assume that for some $U'_i$, $\psi|_{U'_i}\equiv 0$ holds. Consider any other $U'_j$,
where $\psi$ is similar to the holomorphic map $\sigma_j$, with nontrivial overlap $U'_i\cap U'_j\neq\emptyset$.
Then $\sigma_j$ vanishes identically on $U'_i\cap U'_j$ and thus, by the identity theorem for holomorphic
maps, on $U'_j$. The same then applies to $\psi$ and, as a consequence, we obtain $\psi\equiv 0$ on $\Sigma$.
If, on the other hand, $\psi|_{U'_i}$ does not vanish identically for every $U'_i$, the zeros of the corresponding
holomorphic maps $\sigma_i$ are isolated and thus finite, and the same then applies to $\psi$.
\end{proof}

We close this section with elliptic bootstrapping and consequential propositions about
regularity and compactness of holomorphic supercurves.
The next lemma is borrowed from App. B in \cite{MS04}. It is proved by means of an inequality due to
Calderon and Zygmund (\cite{CZ52}, \cite{CZ56}).

\begin{Lem}[Elliptic Bootstrapping]
\label{lemEllipticBootstrapping}
Let $U'\subseteq U\subseteq\bC$ be open sets such that $\overline{U'}\subseteq U$, $l$ be a
positive integer and $p>2$. Then, for every constant $c_0>0$, there is a constant $c>0$ such that
the following holds.
Assume $J\in W^{l,p}(U,\bR^{2n\times 2n})$ satisfies $J^2=-\id$ and $\norm{J}_{U,l,p}\leq c_0$.
Let $k\in\{0,\ldots,l\}$ and $\eta\in W^{k,p}_{\mathrm{loc}}(U,\bR^{2n})$.
Then, every $u\in L^p_{\mathrm{loc}}(U,\bR^{2n})$ that satisfies
\begin{align*}
\partial_su+J\partial_tu=\eta
\end{align*}
is of regularity class $W^{k+1,p}_{\mathrm{loc}}(U,\bR^{2n})$,
and the following estimate holds.
\begin{align*}
\norm{u}_{U',k+1,p}\leq c\left(\norm{\partial_su+J\partial_tu}_{U,k,p}+\norm{u}_{U,k,p}\right)
\end{align*}
\end{Lem}

\begin{Prp}[Regularity]
\label{prpEllipticRegularity}
Let $(A,J)$ be a connection and an almost complex structure on $X$, both of regularity class $C^{l+1}$,
and let $(\varphi,\psi):(\Sigma,L)\rightarrow X$ be an $(A,J)$-holomorphic supercurve of regularity class $W^{1,p}$
such that $p>2$. Then $(\varphi,\psi)$ are of class $W^{l+1,p}$. In particular, $(\varphi,\psi)$ are of class
$C^l$ and, if $(A,J)$ are smooth, then so are $(\varphi,\psi)$.
\end{Prp}

\begin{proof}
It suffices to consider $(\varphi,\psi)$ as a local holomorphic supercurve as in
Lem. \ref{lemLocalHolomorphicSupercurve}, on a coordinate chart $U$. By assumption, it follows that
$\tJ\in W^{1,p}(U,\bR^{2n\times 2n})$ and, applying Lem. \ref{lemEllipticBootstrapping} with
(\ref{eqnLocalHolomorphicCurve}) and $\eta:=0$, we obtain $\varphi\in W^{2,p}_{\mathrm{loc}}(U,\bR^{2n})$
from which, in turn, $\tJ\in W^{2,p}_{\mathrm{loc}}(U,\bR^{2n\times 2n})$ follows.
Continue by induction to obtain $\varphi\in W^{k+1,p}_{\mathrm{loc}}(U,\bR^{2n})$ for $k=1,\ldots,l+1$
and, therefore, $\varphi\in W^{l+2,p}_{\mathrm{loc}}(U,\bR^{2n})$.

It follows that $\tJ\in W^{l+1,p}_{\mathrm{loc}}(U,\bR^{2n\times 2n})$ and, from the construction of $\tD$,
that $\tD\in W^{l,p}_{\mathrm{loc}}(U,\bR^{2n\times 2n})$.
Applying Lem. \ref{lemEllipticBootstrapping} with (\ref{eqnLocalHolomorphicSupercurve})
and $\eta:=-\tilde{D}\cdot\psi\in W^{1,p}_{\mathrm{loc}}(U,\bR^{2n})$, we obtain
$\psi\in W^{2,p}_{\mathrm{loc}}(U,\bR^{2n})$
from which, in turn, $\eta\in W^{2,p}_{\mathrm{loc}}(U,\bR^{2n\times 2n})$ follows.
Continue by induction to obtain $\psi\in W^{k+1,p}_{\mathrm{loc}}(U,\bR^{2n})$ for $k=1,\ldots,l$
and, therefore, $\psi\in W^{l+1,p}_{\mathrm{loc}}(U,\bR^{2n})$.

By the Sobolev estimate (\ref{eqnMorrey}), $(\varphi,\psi)$ are in particular of class $C^l$.
\end{proof}

\begin{Prp}[Compactness]
\label{prpPhiPsiConverging}
Let $(A,J)$ be a connection and an almost complex structure on $X$, both of regularity class $C^{l+1}$,
and let $(A^{\nu},J^{\nu})$ be a sequence of such objects that converges to $(A,J)$ in the
$C^{l+1}$-topology. Let $j^{\nu}$ be a sequence of complex structures on $\Sigma$
converging to $j$ in the $C^{\infty}$-topology, and let $U^{\nu}\subseteq\Sigma$ be an
increasing sequence of open sets whose union is $\Sigma$. Let $(\varphi^{\nu},\psi^{\nu})$
be a sequence of $(A^{\nu},J^{\nu})$-holomorphic supercurves
\begin{align*}
\varphi^{\nu}\in W^{1,p}(U^{\nu},X)\;,\qquad
\psi^{\nu}\in W^{1,p}(U^{\nu},L\otimes_{J^{\nu}}(\varphi^{\nu})^*TX)
\end{align*}
such that $p>2$ and assume that, for every compact set $Q\subseteq\Sigma$,
there exists a compact set $K\subseteq X$ and a constant $c>0$ such that
\begin{align*}
\norm{d\varphi^{\nu}}_{Q,p}\leq c\;,\qquad\varphi^{\nu}(Q)\subseteq K\;,\qquad
\norm{\psi^{\nu}}_{Q,p}\leq c
\end{align*}
for $\nu$ sufficiently large.
Then there exists a subsequence of $(\varphi^{\nu},\psi^{\nu})$, which converges in
the $C^l$-topology on every compact subset of $\Sigma$.
\end{Prp}

\begin{proof}
By an elliptic bootstrapping argument, using Lem. \ref{lemEllipticBootstrapping}, one
shows, consecutively, that a subsequence of $\varphi^{\nu}$ converges in the $C^{l+1}$-topology,
and that a (further) subsequence of $\psi^{\nu}$ converges in the $C^l$-topology.
Both steps are similar to each other and, since the statement for $\varphi^{\nu}$ is well-known,
we prove it only for $\psi^{\nu}$.

Hence, we may assume that $\varphi^{\nu}$ converges to a map $\varphi\in C^{l+1}(\Sigma,X)$
in the $C^{l+1}$-topology. Let $U\subseteq\Sigma$ be sufficiently small such that
$L|_U$ is trivial and $\varphi(U)$ is contained in a coordinate chart of $X$. Then,
for large $\nu$, we may consider $(\varphi^{\nu},\psi^{\nu})$ as a local holomorphic supercurve
on $U$. By Prp. \ref{prpEllipticRegularity}, it follows that
$(\varphi^{\nu},\psi^{\nu})\in W^{l+1,p}_{\mathrm{loc}}(U,\bR^{2n})$ holds,
and Lem. \ref{lemEllipticBootstrapping} yields the following estimate
for $0\leq k\leq l$ and compact subsets $Q_{k+1}\subseteq Q_k\subseteq U$
such that $Q_{k+1}\subseteq\mathrm{int}(Q_k)$.
\begin{align*}
\norm{\psi^{\nu}}_{Q_{k+1},k+1,p}
&\leq c^{\nu}\left(\norm{\partial_s\psi^{\nu}+\tJ^{\nu}\cdot\partial_t\psi^{\nu}}_{Q_k,k,p}
+\norm{\psi^{\nu}}_{Q_k,k,p}\right)\\
&=c^{\nu}\left(\norm{\tD^{\nu}\psi^{\nu}}_{Q_k,k,p}+\norm{\psi^{\nu}}_{Q_k,k,p}\right)\\
&\leq c\left(\norm{\tD^{\nu}}_{U,\infty}\norm{\psi^{\nu}}_{Q_k,k,p}+\norm{\psi^{\nu}}_{Q_k,k,p}\right)\\
&=c^{\nu}\left(\norm{\tD^{\nu}}_{U,\infty}+1\right)\norm{\psi^{\nu}}_{Q_k,k,p}\\
&=:C^{\nu}\norm{\psi^{\nu}}_{Q_k,k,p}
\end{align*}
where $\tJ^{\nu}:=J\circ\varphi^{\nu}$ and $\tD^{\nu}$ are as
in Def. \ref{defLocalHolomorphicSupercurve}, and $c^{\nu}$
denotes the constant from Lem. \ref{lemEllipticBootstrapping} which depends on $\tJ^{\nu}$.
By convergence of $\varphi^{\nu}$, $C^{\nu}$ has a uniform upper bound $C<\infty$ and,
therefore, the norm
\begin{align*}
\sup_{\nu}\norm{\psi^{\nu}}_{Q_{l+1},l+1,p}\leq C\sup_{\nu}\norm{\psi^{\nu}}_{Q_l,l,p}
\leq\ldots\leq C^{l+1}\sup_{\nu}\norm{\psi^{\nu}}_{Q_0,p}<\infty
\end{align*}
is uniformly bounded.
Since the inclusion $W^{l+1,p}(Q)\rightarrow C^l(Q)$ is compact
(provided that the boundary of $Q$ is Lipschitz), this proves the existence of a
subsequence which converges in the $C^l$-topology.
\end{proof}

\begin{Rem}
Convergence of $\psi^{\nu}$ may also be understood in a global fashion as follows.
If $\varphi^{\nu}\rightarrow\varphi$ converges, we may identify $\psi^{\nu}$,
for $\nu$ sufficiently large, with a section $\psi^{\nu}\in W^{1,p}(U^{\nu},L\otimes_J\varphi^*TX)$
via the bundle trivialisation from Lem. \ref{lemGBanachBundle} below.
\end{Rem}

\section{Moduli Spaces and Transversality}
\label{secModuliSpaces}

We examine the moduli spaces of $(A,J)$-holomorphic supercurves.
Let $(X,\omega)$ be a compact symplectic manifold and denote by $\mJ:=\mJ(X,\omega)$
either the space of (smooth) $\omega$-tame or the space of $\omega$-compatible almost
complex structures on $X$. Moreover, let $\mA:=\mA(GL(X))$ be the space
of connections and denote the corresponding spaces of regularity class $C^l$ by
$\mJ^l$ and $\mA^l$, respectively.
Let $\mM^*(\beta,\Sigma;J)$ be as in (\ref{eqnModuliSpaceSimpleCurves}).

\begin{Def}
Let $(A,J)\in\mA\times\mJ$ and $\beta\in H_2(X,\bZ)$. Then we denote the \emph{moduli space
of (nontrivial) $(A,J)$-holomorphic supercurves representing $\beta$} by
\begin{align*}
&\hat{\mM}^*(\beta,\Sigma,L;A,J)\\
&\qquad\qquad:=\{(\varphi,\psi)\in \mM^*(\beta,\Sigma;J)\times
\Omega^0_*(\Sigma,L\otimes_J\varphi^*TX)\setsep\mD^{A,J}_{\varphi}\psi=0\}
\end{align*}
where $\Omega^0_*:=\Omega^0\setminus\{0\}$ is the space of nontrivial sections.
\end{Def}

The focus of this section is to show that the moduli spaces thus defined are oriented smooth manifolds,
provided that $(A,J)$ is chosen generic.
Here, it is important to exclude the zero sections $\psi\equiv 0$ for, otherwise,
our transversality argument below breaks down. En passant, we will see that it does not suffice
to vary only $J$ but we need to enlarge the parameter space towards pairs $(A,J)$ as incorporated in Def.
\ref{defAJHolomorphicSupercurve}. Moreover, one encounters the usual problem that
spaces of smooth objects are, in general, not Banach manifolds but only Fréchet manifolds. There are two standard
approaches to resolve this issue: One due to Floer (cf. \cite{Flo88}),
that stays within the smooth category, and a second approach (as presented e.g. in Chp. 3 of \cite{MS04})
which we prefer here, that deals with objects of lower regularity in the first place and later shows independence of
the exact Sobolev spaces involved.
We introduce the operator
\begin{align*}
&\hat{\mD}_{\varphi,\psi}^{A,J}:\Omega^0(\Sigma,\varphi^*TX)\oplus\Omega^0(\Sigma,L\otimes_J\varphi^*TX)\\
&\qquad\qquad\rightarrow\Omega^{0,1}(\Sigma,\varphi^*TX)\oplus\Omega^{0,1}(\Sigma,L\otimes_J\varphi^*TX)
\end{align*}
by prescribing
\begin{align}
\label{eqnVerticalDifferential}
\hat{\mD}_{\varphi,\psi}^{A,J}[\Xi+\Psi]
:=D_{\varphi}\Xi+\mD_{\varphi}^{A,J}\Psi+\nabla_{\Xi}^{A,J}\left(\mD_{\varphi}^{A,J}\right)\psi
\end{align}
where, by definition, $\nabla_{\Xi}^{A,J}\left(\mD_{\varphi}^{A,J}\right)\psi
=\nabla_{\Xi}^{A,J}\left(\mD_{\varphi}^{A,J}\psi\right)-\mD_{\varphi}^{A,J}\nabla_{\Xi}^{A,J}\psi$.
We will see in Sec. \ref{subsecProofsTransversality} that this is exactly the vertical differential
$\hat{\mD}_{\varphi,\psi}^{A,J}=D_{(\varphi,\psi)}\fF$ of a map $\fF$ such that
$\hat{\mM}^*(\beta,\Sigma,L;A,J)=\fF^{-1}(0)$.

\begin{Def}
\label{defSuperregular}
A pair $(A,J)\in\mA\times\mJ$ is called \emph{regular} (for $\beta$, $\Sigma$ and $L$)
if $\hat{\mD}_{\varphi,\psi}^{A,J}$ is surjective for every $(\varphi,\psi)\in\hat{\mM}^*(\beta,\Sigma,L;A,J)$.
We denote by $\mA\mJ_{\mathrm{reg}}:=\mA\mJ_{\mathrm{reg}}(\beta,\Sigma,L)$ the subset of all
$(A,J)\in\mA\times\mJ$ that are regular for $\beta$, $\Sigma$ and $L$.
\end{Def}

\begin{Thm}
\label{thmSuperregular}
$ $
\begin{enumerate}
\renewcommand{\labelenumi}{(\roman{enumi})}
\item If $(A,J)\in\mA\mJ_{\mathrm{reg}}(\beta,\Sigma,L)$, then the moduli space $\hat{\mM}^*(\beta,\Sigma,L;A,J)$ is a smooth
manifold of finite dimension
\begin{align*}
\mathrm{ind}_{\bR}D_{\varphi}+\mathrm{ind}_{\bR}\mD_{\varphi}^{A,J}=n(4-4g)+4c_1(\beta)+2n\deg L
\end{align*}
It carries a natural orientation.

\item The set $\mA\mJ_{\mathrm{reg}}(\beta,\Sigma,L)$ is of the second category (in the sense of Baire)
in $\mA\times\mJ$.
\end{enumerate}
\end{Thm}

This is our main theorem. We shall delay the proof to subsection \ref{subsecProofsTransversality} below.
It does not imply that the moduli spaces are non-empty.
Existence of holomorphic curves (with trivial homology class) is clear: If $\varphi:\Sigma\rightarrow X$ is constant,
it does in particular satisfy $\dbar_J\varphi=0$. The trivial case for $\psi\in\ker\mD_{\varphi}^{A,J}$
is $\psi\equiv 0$. This solution does always exist, but holomorphic curves $(\varphi,0)$ are excluded
from the moduli spaces. The following proposition yields a sufficient condition for nontrivial
holomorphic supercurves to exist.

\begin{Prp}
\label{prpExistenceOfSolutions}
Let $\varphi_0\in\mM^*(\beta,\Sigma;J)$ be a holomorphic curve. Then the set $V_{\varphi_0}$ of $(A,J)$-holomorphic
supercurves $(\varphi,\psi)$ such that $\varphi=\varphi_0$ is a complex vector space with dimension
at least
\begin{align*}
\dim_{\bR}V_{\varphi_0}\geq n(2-2g)+2c_1(\beta)+2n\deg L
\end{align*}
In particular, non-trivial holomorphic supercurves exist if the right hand side is positive.
\end{Prp}

For the important case $\Sigma=S^2$ (i.e. $g=0$) and $L=S^+$ (i.e. $\deg L=-1$), the condition
thus becomes $\dim_{\bR}V_{\varphi_0}\geq 2c_1(\beta)$. The (easy) proof of Prp. \ref{prpExistenceOfSolutions}
will be given in the next subsection.
Although the spaces $V_{\varphi}$ are linear, it is not true, in general, that the moduli space
$\hat{\mM}^*(\beta,\Sigma,L;A,J)$ is a vector bundles over the classical moduli space $\mM^*(\beta,\Sigma;J)$,
not even for $(A,J)$ regular. This is because, for $\psi=0$, the transversality argument used in
the proof of Thm. \ref{thmSuperregular} below breaks down, so there may be singularities.

\subsection{Fredholm Theory}

We study the operators $\mD_{\varphi}^{A,J}$ and $\hat{\mD}_{\varphi,\psi}^{A,J}$.
Let us first recall some properties of Cauchy-Riemann operators of lower regularity.
Let $E\rightarrow\Sigma$ be a complex vector bundle of regularity class $W^{l,p}$.
For everything to be well-defined it is important here that $p>2$ (see the discussion in
Sec. \ref{secProperties} above). We say that a \emph{real (complex) linear Cauchy-Riemann operator
of class $W^{l-1,p}$} on $E$ is an $\bR$-linear ($\bC$-linear) operator
\begin{align*}
D:W^{l,p}(\Sigma,E)\rightarrow W^{l-1,p}(\Sigma,\Lambda^{0,1}\Sigma\otimes_{\bC}E)
\end{align*}
satisfying the Leibniz rule $D_{\varphi}(f\xi)=f(D_{\varphi}\xi)+(\dbar f)\xi$ for all $f\in W^{l,p}(\Sigma,\bR)$
(or, respectively, $f\in W^{l,p}(\Sigma,\bC)$).
Moreover, we say that a connection $\nabla$ is of regularity class $W^{l-1,p}$ if the corresponding connection potential
$A$ is of this class. Analogous to the smooth case,
any connection $\nabla$ on $E$ of regularity class $W^{l-1,p}$ gives rise to a real linear Cauchy-Riemann
operator via $D:=\nabla^{0,1}$.
It is further clear that the sum $D+\alpha$ of a real linear Cauchy-Riemann operator $D$ with some
$\alpha\in W^{l-1,p}(\Sigma,\Lambda^{0,1}\Sigma\otimes_{\bR}\mathrm{End}_{\bR}(E))$
is again a real linear Cauchy-Riemann operator.

Let $\bigwedge^k\Sigma:=\bigwedge^kT^*\Sigma$ be the bundle of complex valued $k$-forms on $\Sigma$
and $\scal{\cdot}{\cdot}$ denote either a Hermitian bundle metric on $E$ or
the bundle metric on $\bigwedge^k\Sigma\otimes_{\bC}E$, induced by the former and
a Riemann metric on $\Sigma$ in the conformal class corresponding to $j$.
Such a bundle metric induces $L^2$ inner products $\int_{\Sigma}\dvol_{\Sigma}\scal{\cdot}{\cdot}$ on the spaces
of $E$-valued $k$-forms on $\Sigma$ with respect to which the formal adjoint
$D^*:W^{l,p}(\Sigma,\Lambda^{0,1}\Sigma\otimes_{\bC}E)\rightarrow W^{l-1,p}(\Sigma,E)$
of a real linear Cauchy-Riemann operator $D$ of class $W^{l-1,p}$ is defined.
For a proof of the next theorem, consult e.g. Appendix C in \cite{MS04}.

\begin{Prp}[Riemann-Roch Theorem]
\label{prpRiemannRoch}
Let $D$ be a real linear Cauchy-Riemann operator on $E$ of class $W^{l-1,p}$ with
$p>2$. Then the following holds for every $k\in\{1,\ldots,l\}$ and $q>1$ such that $1/p+1/q=1$.
\begin{enumerate}
\renewcommand{\labelenumi}{(\roman{enumi})}
\item The operators $D:W^{k,p}(\Sigma,E)\rightarrow W^{k-1,p}(\Sigma,\Lambda^{0,1}\Sigma\otimes_{\bC}E)$ and
$D^*:W^{k,p}(\Sigma,\Lambda^{0,1}\Sigma\otimes_{\bC}E)\rightarrow W^{k-1,p}(\Sigma,E)$
are Fredholm with index
\begin{align*}
\mathrm{ind}_{\bR}D=-\mathrm{ind}_{\bR}D^*=n(2-2g)+2\scal{c_1(E)}{[\Sigma]}
\end{align*}
where $g$ denotes the genus of $\Sigma$ and $n$ is the (complex) rank of $E$.

\item If $\eta\in L^q(\Sigma,\Lambda^{0,1}\Sigma\otimes_{\bC}E)$ satisfies
$\int_{\Sigma}\scal{\eta}{D\xi}\dvol_{\Sigma}=0$
for every $\xi\in W^{k,p}(\Sigma,E)$ then $\eta\in W^{k,p}(\Sigma,\Lambda^{0,1}\Sigma\otimes_{\bC}E)$ and $D^*\eta=0$.

\item If $\xi\in L^q(\Sigma,E)$ satisfies $\int_{\Sigma}\scal{\xi}{D^*\eta}\dvol_{\Sigma}=0$ for every\\
$\eta\in W^{k,p}(\Sigma,\Lambda^{0,1}\Sigma\otimes_{\bC}E)$ then $\xi\in W^{k,p}(\Sigma,E)$ and $D\xi=0$.
\end{enumerate}
\end{Prp}

Prp. \ref{prpRiemannRoch} shows, in particular, that every element in the kernel of $D$ for $k\leq l$
is automatically of class $W^{l,p}$ for any $p>2$, and so the kernel of $D$ does not depend on the precise
choice of the space on which the operator is defined. The same is true for the cokernel. In particular,
$D$ is surjective for some choice of $k$ and $p$ if and only if it is surjective for all such choices.

The following two lemmas follow from Prp. \ref{prpRiemannRoch} applied to the complex bundle
$E:=L\otimes_J\varphi^*TX$. Note that the regularity of $E$ is determined by $\varphi$. More precisely,
a simple consideration concerning transition maps for the pullback bundle $\varphi^*TX$ shows that
$E$ is of class $W^{l,p}$ if $\varphi$ is (for $p>2$).

\begin{Lem}
\label{lemComplexCauchyRiemann}
Let $l$ be a positive integer and $p>2$. Let $(A,J)\in\mA^{l+1}\times\mJ^{l+1}$
and $\varphi\in W^{l+1,p}$. Then $\mD_{\varphi}^{A,J}$ is a complex linear Cauchy-Riemann operator
of regularity class $W^{l,p}$ with Fredholm index
\begin{align*}
\mathrm{ind}_{\bR}\mD_{\varphi}^{A,J}=n(2-2g)+2c_1(\beta)+2n\deg L
\end{align*}
where, as usual, we denote $\beta:=[\varphi]\in H_2(X;\bZ)$.
\end{Lem}

\begin{proof}
By the local form of $(\nabla^{A,J})^{0,1}$ in (\ref{eqnLocalComplexConnection}), $\varphi$ and $J$ enter
with their first derivatives, so the operator is of class $W^{l,p}$. Like in the smooth case,
it is clear that $\mD_{\varphi}^{A,J}$ is still a
complex linear Cauchy-Riemann operator and, as such, satisfies the hypotheses of Prp. \ref{prpRiemannRoch}.
Its Fredholm index
\begin{align*}
\mathrm{ind}_{\bR}\mD_{\varphi}^{A,J}&=n(2-2g)+2\scal{c_1(L\otimes_J\varphi^*TX)}{[\Sigma]}\\
&=n(2-2g)+2\scal{1\cdot c_1(\varphi^*TX)}{[\Sigma]}+2\scal{n\cdot c_1(L)}{[\Sigma]}\\
&=n(2-2g)+2\scal{c_1(TX)}{[\varphi_*\Sigma]}+2n\scal{c_1(L)}{[\Sigma]}
\end{align*}
is easily calculated, using properties of the first Chern class.
\end{proof}

\begin{proof}[Proof of Prp. \ref{prpExistenceOfSolutions}]
$V_{\varphi}=\ker\mD_{\varphi}^{A,J}$ is clearly a complex vector space since
the operator is complex linear. By definition, the dimension of the kernel of a Fredholm operator is
greater than or equal to its Fredholm index (with equality in the surjective case).
\end{proof}

The expected dimensions of the moduli spaces stated in Thm. \ref{thmSuperregular} can be understood as follows.

\begin{Lem}
\label{lemCROperatorFredholmSum}
Let $l$ be a positive integer and $p>2$. Let $(A,J)\in\mA^{l+1}\times\mJ^{l+1}$
and $\varphi,\psi\in W^{l+1,p}$. Then $\hat{\mD}_{\varphi,\psi}^{A,J}$ is a real linear Cauchy-Riemann operator
of regularity class $W^{l,p}$ with Fredholm index
\begin{align*}
\mathrm{ind}_{\bR}\hat{\mD}_{\varphi,\psi}^{A,J}
=\mathrm{ind}_{\bR}D_{\varphi}+\mathrm{ind}_{\bR}\mD_{\varphi}^{A,J}=n(4-4g)+4c_1(\beta)+2n\deg L
\end{align*}
\end{Lem}

\begin{proof}
Since $D_{\varphi}$ and $\mD_{\varphi}^{A,J}$ are defined on different spaces, it is clear that
their sum $D_{\varphi}+\mD_{\varphi}^{A,J}$ is a complex linear Cauchy-Riemann operator of class $W^{l,p}$
since, by the properties of $D_{\varphi}$ and Lem. \ref{lemComplexCauchyRiemann},
the same applies to each summand individually.
Moreover, the map $\Xi\mapsto\nabla_{\Xi}^{A,J}\left(\mD_{\varphi}^{A,J}\right)\psi$ is of class $W^{l,p}$ and
linear for $\bR$-valued functions, and thus it follows that $\hat{\mD}_{\varphi,\psi}^{A,J}$ is a real linear
Cauchy-Riemann operator of regularity class $W^{l,p}$. By Prp. \ref{prpRiemannRoch} applied to
$E:=\varphi^*TX\oplus(L\otimes_J\varphi^*TX)$, its Fredholm index coincides
with the Fredholm index of $D_{\varphi}+\mD_{\varphi}^{A,J}$, which is the sum of the individual indices,
given by (\ref{eqnDimensionClassicalModuliSpace}) and Lem \ref{lemComplexCauchyRiemann}, respectively.
\end{proof}

\subsection{The Universal Moduli Space}

In this subsection, we introduce the universal moduli space of connections, almost complex structures and
holomorphic supercurves and prove that, in the Banach manifold setting, it is a separable manifold.
It will turn out to be essential in the proof of the second part of Thm. \ref{thmSuperregular}.
For $(A,J)\in\mA^{l+1}\times\mJ^{l+1}$, we denote by
\begin{align*}
&\hat{\mM}^*(\beta,\Sigma,L;A,J)\\
&\qquad\qquad:=\{(\varphi,\psi)\in \mM^*_{C^l}(\beta,\Sigma;J)
\times C^l_*(\Sigma,L\otimes_J\varphi^*TX)\setsep\mD^{A,J}_{\varphi}\psi=0\}
\end{align*}
the moduli space of (nontrivial) $(A,J)$-holomorphic supercurves of class $C^l$. Note that, by elliptic regularity as stated
in Prp. \ref{prpEllipticRegularity}, $\hat{\mM}^*(\beta,\Sigma,L;A,J)$ may be considered as a space of objects
of regularity $W^{k,p}$ for $k\leq l+1$ and $p>2$.

\begin{Def}
The \emph{universal moduli space} of holomorphic supercurves is defined as
\begin{align*}
&\hat{\mM}^*(\beta,\Sigma,L;\mA^{l+1},\mJ^{l+1})\\
&\qquad:=\{(\varphi,\psi,A,J)\setsep A\in\mA^{l+1},\,J\in\mJ^{l+1},\,
(\varphi,\psi)\in\hat{\mM}^*(\beta,\Sigma,L;A,J)\}
\end{align*}
\end{Def}

For the proof of Prp. \ref{prpUniversalModuliSpace} below, stating that the universal moduli space
is a Banach manifold, we need the bundles to be introduced in the following
two lemmas. We refer to \cite{Lan99} for details on Banach space bundles. Let
\begin{align*}
\mB^{k,p}_{\beta}:=\{\varphi\in W^{k,p}(\Sigma,X)\setsep [\varphi]=\beta\}
\end{align*}
denote the space of all $W^{k,p}$-maps $\varphi:\Sigma\rightarrow X$ such that $[\varphi]=\beta\in H_2(X,\bZ)$,
which is well-defined provided that $kp>2$ (cf. the discussion in Sec. \ref{secProperties}).
The space $\mB^{k,p}_{\beta}$ is a component of $W^{k,p}(\Sigma,X)$ and, as such,
a smooth separable Banach manifold with tangent space $T_{\varphi}\mB^{k,p}_{\beta}=W^{k,p}(\Sigma,\varphi^*TX)$
as shown in \cite{Eli67}.

\begin{Lem}
\label{lemGBanachBundle}
Let $1\leq k\leq l$ be natural numbers, $p>2$ and $\beta\in H_2(X,\bZ)$. Then, prescribing
\begin{align*}
\mG^{k,p}_{(\varphi,A,J)}:=W^{k,p}(\Sigma,L\otimes_J\varphi^*TX)
\end{align*}
for $\varphi\in\mB^{k,p}_{\beta}$, $A\in\mA^{l+1}$ and $J\in\mJ^{l+1}$, defines a separable
$C^{l-k}$-Banach space bundle $\mG^{k,p}\rightarrow\mB^{k,p}_{\beta}\times\mA^{l+1}\times\mJ^{l+1}$.
\end{Lem}

\begin{proof}
Identify a neighbourhood $U(\varphi)$ of a smooth function $\varphi:\Sigma\rightarrow X$
with a neighbourhood of zero in $W^{k,p}(\Sigma,\varphi^*TX)$ via $\xi\mapsto\exp_{\varphi}(\xi)$,
where the exponential function is with respect to the Levi-Civita connection of any fixed smooth metric on $X$.
The bundle trivialisation over such a coordinate charts comes in three stages:
\begin{itemize}
\item Fix $A\in\mA^{l+1}$ and $J\in\mJ^{l+1}$. We trivialise $\mG^{k,p}$ over
$U(\varphi)\times\{A\}\times\{J\}$ via $\nabla^{A,J}$-parallel transport
\begin{align*}
\mP^{A,J}_{\varphi}(\xi):L\otimes_J\varphi^*TX\rightarrow L\otimes_J\exp_{\varphi}(\xi)^*TX
\end{align*}
Consider another such trivialisation over $U(\tvarphi)\times\{A\}\times\{J\}$ with
$\tvarphi=\exp_{\varphi}(\xi)\in U(\varphi)$. The corresponding transition map is then given by
$\xi\mapsto\mP^{A,J}_{\varphi}(\xi)\circ\xi$ which maps
\begin{align*}
W^{k,p}(\Sigma,L\otimes_J\varphi^*TX)\rightarrow W^{k,p}(\Sigma,L\otimes_J\tvarphi^*TX)
\end{align*}
Since $A$ and $J$ are of class $C^{l+1}$, parallel transport $\mP^{A,J}_{\varphi}(\tilde{\xi})$
is of class $C^{l}$ and, upon differentiating $l-k$ times, still maps a $W^{k,p}$-section to a
well-defined $W^{k,p}$-section. This shows that the induced transition maps, which we denote
\begin{align*}
\mP^{A,J}_{\varphi}(\xi):\mG^{k,p}_{(\varphi,A,J)}\rightarrow\mG^{k,p}_{(\exp_{\varphi}(\xi),A,J)}
\end{align*}
are of regularity class $C^{l-k}$.

\item Now we trivialise $\mG^{k,p}$ over neighbourhoods $\{\varphi\}\times U(A)\times\{J\}$
via the (smooth) isomorphisms
\begin{align*}
&I_{\varphi}^J(A,A'):\mG^{k,p}_{(\varphi,A,J)}\rightarrow\mG^{k,p}_{(\varphi,A',J)}\\
&I_{\varphi}^J(A,A'):=\id_{W^{k,p}(\Sigma,L\otimes_J\varphi^*TX)}:\psi\mapsto\psi
\end{align*}

\item Finally, we trivialise $\mG^{k,p}$ over neighbourhoods $\{\varphi\}\times\{A\}\times U(J)$
via the (smooth) isomorphisms
\begin{align*}
L_{\varphi}^A(J,J'):L\otimes_J\varphi^*TX\rightarrow L\otimes_{J'}\varphi^*TX
\end{align*}
defined by $\psi\mapsto\frac{1}{2}(\psi-iJ'\circ\psi)$
with $i$ the complex structure on $L$.
\end{itemize}
The entire bundle trivialisation is now composed of
\begin{align*}
&\mT_{(\varphi,A,J)}^{(\exp_{\varphi}(\xi),A',J')}\\
&\qquad:=\mP^{A',J'}_{\varphi}(\xi)\circ L_{\varphi}^{A'}(J,J')\circ I_{\varphi}^J(A,A'):
\mG^{k,p}_{(\varphi,A,J)}\rightarrow\mG^{k,p}_{(\exp_{\varphi}(\xi),A',J')}
\end{align*}
Thus we see that $\mG^{k,p}$ is a Banach space bundle of class $C^{l-k}$.
Since $\mB^{k,p}_{\beta}$, $\mA^{l+1}$ and $\mJ^{l+1}$ are all separable, one finds an atlas
of $\mB^{k,p}_{\beta}\times\mA^{l+1}\times\mJ^{l+1}$ with countably many charts
$U_{\alpha}=U(\varphi_{\alpha})\times U(A_{\alpha})\times U(J_{\alpha})$,
locally trivialising $\mG^{k,p}$ such that
$\mG^{k,p}|_{U_{\alpha}}\cong U_{\alpha}\times W^{k,p}(\Sigma,L\otimes_{J_{\alpha}}\varphi_{\alpha}^*TX)$.
Since also the spaces $W^{k,p}(\Sigma,L\otimes_{J_{\alpha}}\varphi_{\alpha}^*TX)$ are separable, every
$\mG^{k,p}|_{U_{\alpha}}$ is separable and, therefore, so is $\mG^{k,p}$ (being the countable union of
separable sets).
\end{proof}

\begin{Lem}
\label{lemEBanachBundle}
Let $1\leq k\leq l$ be natural numbers, $p>2$ and $\beta\in H_2(X,\bZ)$. Then, prescribing
\begin{align*}
\mE^{k-1,p}_{(\varphi,A,J)}&:=W^{k-1,p}(\Sigma,\Lambda^{0,1}\Sigma\otimes_J\varphi^*TX)\\
\hat{\mE}^{k-1,p}_{(\varphi,A,J)}&:=W^{k-1,p}(\Sigma,\Lambda^{0,1}\Sigma\otimes_{\bC}L\otimes_J\varphi^*TX)
\end{align*}
for $\varphi\in\mB^{k,p}_{\beta}$, $A\in\mA^{l+1}$ and $J\in\mJ^{l+1}$ defines separable
$C^{l-k}$-Banach space bundles
$\mE^{k-1,p},\,\hat{\mE}^{k-1,p}\rightarrow\mB^{k,p}_{\beta}\times\mA^{l+1}\times\mJ^{l+1}$.
\end{Lem}

\begin{proof}
This is completely analogous to Lem. \ref{lemGBanachBundle}. Here, the $U(J)$-trivialising
isomorphism of $\hat{\mE}^{k-1,p}$ is the map
\begin{align*}
\tL_{\varphi}^A(J,J'):\Lambda^{0,1}\Sigma\otimes_{\bC}L\otimes_J\varphi^*TX\rightarrow
\Lambda^{0,1}\Sigma\otimes_{\bC}L\otimes_{J'}\varphi^*TX
\end{align*}
defined by $\alpha\mapsto\frac{1}{2}(\alpha-iJ'\alpha)$
with $i$ the complex structure on $\Lambda^{0,1}\Sigma\otimes_{\bC}L$.
\end{proof}

By the regularity Proposition \ref{prpEllipticRegularity}, every $(A,J)$-holomorphic supercurve
$(\varphi,\psi):(\Sigma,L)\rightarrow X$ of regularity class $W^{k,p}$ is automatically
of class $C^l$ if $k\leq l+1$ and $(A,J)\in C^{l+1}$.
Therefore, we may identify the universal moduli space of class $C^l$ with a subspace of $\mG^{k,p}$ as follows.
For $\psi\in\mG^{k,p}$, we abbreviate $\varphi:=\pi_{\mB}(\psi)$, $A:=\pi_{\mA}(\psi)$ and $J:=\pi_{\mJ}(\psi)$.
Then
\begin{align*}
\hat{\mM}^*(\beta,\Sigma,L;\mA^{l+1},\mJ^{l+1})
=\{\psi\in\mG^{k,p}\setminus\{0\}\setsep\dbar_J\varphi=0\,,\;\mD_{\varphi}^{A,J}\psi=0\}
\end{align*}
Moreover, defining a map $\mF:\mG^{k,p}\setminus\{0\}\rightarrow\mE^{k-1,p}\oplus\hat{\mE}^{k-1,p}$ by
\begin{align}
\label{eqnMF}
\mF(\psi):=\mF_1(\psi)+\mF_2(\psi):=\dbar_J\varphi+\mD_{\varphi}^{A,J}\psi
\end{align}
where $0\in\mG^{k,p}$ denotes the zero section, we may identify
\begin{align}
\label{eqnUniversalModuliSpace}
\hat{\mM}^*(\beta,\Sigma,L;\mA^{l+1},\mJ^{l+1})=\mF^{-1}(0)
\end{align}

We next calculate the vertical differential of $\mF$, that is the differential
of $\mF$ with respect to the above trivialisations of the bundles $\mG^{k,p}$, $\mE^{k-1,p}$ and $\hat{\mE}^{k-1,p}$.
We define $\mathrm{End}(TX,J,\omega)\subseteq\mathrm{End}(TX)$ to be the bundle of
endomorphisms $Y$ such that $YJ+JY=0$ and, if we consider $\omega$-compatible almost complex structures,
such that in addition $\scal[\omega]{Yv}{w}+\scal[\omega]{v}{Yw}=0$ holds for all tangent vectors $v$ and $w$.
These constraints are obtained by differentiating $J^2=-\id$ and the compatibility condition and, therefore,
the tangent space at $J$ is precisely the space of sections
\begin{align*}
T_J\mJ^{l+1}=C^{l+1}(X,\mathrm{End}(TX,J,\omega))
\end{align*}
Let $V\subseteq X$ be a sufficiently small open subset such that $TX|_V\cong\bR^{2n}$ is trivial
and, on $U:=\varphi^{-1}(V)\subseteq\Sigma$, there exists a nonvanishing local section
$\theta\in\Gamma(U,L)$. Then, restricted to $U$, any section $\psi$ of $L\otimes_J\varphi^*TX$ may be identified
with a section $\psi_{\theta}$ of $\varphi^*TX$ via $\psi=\theta\cdot\psi_{\theta}$ and, moreover,
we fix a trivialisation of $TX\cong GL(X)\times_{GL(2n,\bR)}\bR^{2n}$ induced by a section
$\os\in\Gamma(V,GL(X))$. Such a trivialisation determines the action of $\fg=gl(2n,\bR)$ on $TX|_V$.

\begin{Lem}
\label{lemUniversalVerticalDifferential}
Let $l\geq 2$, $p>2$ and $k\in\{1,\ldots,l\}$. Then $\mF$ is a function of class $C^{l-k}$.
Moreover, the vertical differential is the map
\begin{align*}
&D_{(\varphi,\psi,A,J)}(\mF_1+\mF_2):
W^{k,p}(\Sigma,\varphi^*TX)\oplus W^{k,p}(\Sigma,L\otimes_J\varphi^*TX)\\
&\qquad\qquad\qquad\qquad\quad\oplus \Omega^1_{C^{l+1}}(\Sigma,GL(X)\times_{Ad}\fg)\oplus C^{l+1}(X,\mathrm{End}(TX,J,\omega))\\
&\qquad\rightarrow W^{k-1,p}(\Sigma,\Lambda^{0,1}\Sigma\otimes_J\varphi^*TX)\oplus
W^{k-1,p}(\Sigma,\Lambda^{0,1}\Sigma\otimes_{\bC}L\otimes_J\varphi^*TX)
\end{align*}
defined by
\begin{align*}
&D_{(\varphi,\psi,A,J)}\mF_1[\Xi+\Psi+Q+Y]=D_{\varphi}\Xi+\frac{1}{2}Y(\varphi)\circ d\varphi\circ j\\
&D_{(\varphi,\psi,A,J)}\mF_2[\Xi+\Psi+Q+Y]\\
&\qquad\qquad\qquad=\mD_{\varphi}^{A,J}\Psi+\nabla^{A,J}_{\Xi}\left(\mD^{A,J}_{\varphi}\right)\psi
+\left((Q^J\circ d\os)^{0,1}\circ d\varphi\right)\psi\\
&\qquad\qquad\qquad\qquad+\frac{1}{2}JY\circ\mD_{\varphi}^{A,J}\psi
-\frac{1}{2}\mD_{\varphi}^{A,J}\circ JY\psi+D_Y(\varphi,\psi,A,J)[Y]
\end{align*}
where $D_{\varphi}$ denotes the linearisation of $\dbar_J$ as in (\ref{eqnLinearisedDbar}) and
\begin{align*}
&D_Y(\varphi,\psi,A,J)[Y]\\
&\qquad\qquad:=\frac{1}{4}\left(-Y\nabla^A(J\psi_{\theta})-J\nabla^A(Y\psi_{\theta})
+Y\nabla_{j[\cdot]}^A\psi_{\theta}+\nabla_{j[\cdot]}^A(Y\psi_{\theta})\right)\cdot\theta
\end{align*}
The expressions stated are independent of the choices of $\os$ and $\theta$.
\end{Lem}

\begin{proof}
As for $\mF_1$, the statement follows as in Chp. 3 of \cite{MS04}. Therefore, it remains
to consider $\mF_2$. Let $\mU:=U(\varphi)\times U(A)\times U(J)$ be a (sufficiently small) neighbourhood of
$(\varphi,A,J)\in\mB^{k,p}_{\beta}\times\mA^{l+1}\times\mJ^{l+1}$ such that
\begin{align*}
\mG^{k,p}|_{\mU}\cong U(\varphi)\times U(A)\times U(J)\times W^{k,p}(\Sigma,L\otimes_J\varphi^*TX)
\end{align*}
is trivialised under the fibre isomorphisms $\mT_{(\varphi,A,J)}^{(\varphi',A',J')}$,
which are stated explicitely in the proof of Lem. \ref{lemGBanachBundle}.
As mentioned above, the identification of $U(\varphi)$ with a neighbourhood of zero in
$W^{k,p}(\Sigma,\varphi^*TX)$ is given by identifying $\xi$ with $\varphi'=\exp_{\varphi}(\xi)$.
Similarly, we choose $\mU$ sufficiently small such that $\hat{\mE}^{k-1,p}$ is trivialised
via fibre isomorphisms $\tilde{\mT}_{(\varphi,A,J)}^{(\varphi',A',J')}$ which are analogous
to $\mT_{(\varphi,A,J)}^{(\varphi',A',J')}$ (with $L_{\varphi}^A(J,J')$ replaced by
$\tL_{\varphi}^A(J,J')$ as in the proof of Lem. \ref{lemEBanachBundle}).
With respect to these trivialisations, $\mF_2$ becomes the map
\begin{align*}
\mF_2:U(\varphi)&\times W^{k,p}(\Sigma,L\otimes_J\varphi^*TX)\times U(A)\times U(J)\\
&\rightarrow W^{k-1,p}(\Sigma,\Lambda^{0,1}\Sigma\otimes_{\bC}L\otimes_J\varphi^*TX)\\
(\varphi',\psi',A',J')&\mapsto\left(\tilde{\mT}_{(\varphi,A,J)}^{(\varphi',A',J')}\right)^{-1}
\circ\mD_{\varphi'}^{A',J'}\circ\mT_{(\varphi,A,J)}^{(\varphi',A',J')}\circ\psi'
\end{align*}
$\mT$ and $\tilde{\mT}^{-1}$ are known to be of class $C^{l-k}$ and,
by the local formula in (\ref{eqnLocalComplexConnection}), $\mD_{\varphi'}^{A',J'}$
is of class $C^l$ if $(A',J')$ are of class $C^{l+1}$. Therefore,
we see that $\mF_2$ is a $C^{l-k}$ map, and
it remains to calculate the differential of $\mF_2$ in these trivialisations.

We calculate the differential with respect to $\varphi$ or, equivalently, with respect to $\xi$, at $\xi=0$.
Let $\Xi\in W^{k,p}(\Sigma,\varphi^*TX)$. Then
\begin{align*}
D_{(\varphi,\psi,A,J)}\mF_2[\Xi]
&=D_{(\xi=0,\psi,A,J)}\left(\left(\mP_{\varphi}^{A,J}(\xi)\right)^{-1}\circ
\mD^{A,J}_{\exp_{\varphi}\xi}\circ\mP_{\varphi}^{A,J}(\xi)\circ\psi\right)[\Xi]\\
&=\frac{d}{dt}|_0\left(\left(\mP_{\varphi}^{A,J}(t\Xi)\right)^{-1}\circ
\mD^{A,J}_{\exp_{\varphi}t\Xi}\circ\mP_{\varphi}^{A,J}(t\Xi)\right)\circ\psi
\end{align*}
Here, $\mP_{\varphi}^{A,J}(t\Xi)$ denotes parallel transport along the path $t\mapsto\exp_{\varphi}(t\Xi)$,
and thus we calculate the differential as follows.
\begin{align*}
D_{(\varphi,\psi,A,J)}\mF_2[\Xi]&=\nabla^{A,J}_{\frac{d}{dt}|_0(\exp_{\varphi}(t\Xi))}
\left(\mD^{A,J}_{\exp_{\varphi}t\Xi}\right)\circ\psi\\
&=\nabla^{A,J}_{d_0\exp_{\varphi}[\Xi]}\left(\mD^{A,J}_{\varphi}\right)\psi
=\nabla^{A,J}_{\Xi}\left(\mD^{A,J}_{\varphi}\right)\psi
\end{align*}
For $\Psi\in W^{k,p}(\Sigma,L\otimes_J\varphi^*TX)$, we similarly obtain
\begin{align*}
D_{(\varphi,\psi,A,J)}\mF_2[\Psi]
=D_{(\varphi,\psi,A,J)}\left(\mD_{\varphi}^{A,J}\psi\right)[\Psi]
=\mD_{\varphi}^{A,J}\Psi
\end{align*}
which is the differential with respect to $\psi$.
Next, we calculate the differential with respect to $A$.
Let $Q\in T_A\mA=\Omega^1_{C^{l+1}}(\Sigma,GL(X)\times_{Ad}\fg)$. Then
\begin{align*}
D_{(\varphi,\psi,A,J)}\mF_2[Q]
&=D_{(\varphi,\psi,A,J)}\left(\mD_{\varphi}^{A,J}\psi\right)[Q]\\
&=D_{(\varphi,\psi,A,J)}\left(\left(\nabla^{A,J}\psi_{\theta}\right)^{0,1}\cdot\theta
+\psi_{\theta}\cdot(\dbar\theta)\right)[Q]\\
&=D_{(\varphi,\psi,A,J)}\left(\left(\nabla^{A,J}\psi_{\theta}\right)^{0,1}\right)[Q]\cdot\theta\\
&=\left(D_{(\varphi,\psi,A,J)}\left(\nabla^{A,J}\psi_{\theta}\right)[Q]\right)^{0,1}\cdot\theta\\
\end{align*}
As in (\ref{eqnLocalComplexConnection}), we (locally) identify
$\psi_{\theta}=[s,v]$ with $s:=\os\circ\varphi$ and $v\in W^{k,p}(U,\bR^{2n})$
and prescribe $Q^J:=\frac{1}{2}(Q-J\cdot Q\cdot J)$, where
$J$ is identified with a matrix-valued function. Then
\begin{align*}
&D_{(\varphi,\psi,A,J)}\mF_2[Q]\\
&\qquad=\left(D_{(\varphi,\psi,A,J)}\left(\left[s\,,\;dv[\cdot]
+(C\circ d\varphi)\cdot v\right]\right)[Q]\right)^{0,1}\cdot\theta\\
&\qquad=\left(D_{(\varphi,\psi,A,J)}\left(\left[s\,,\;
\left(\frac{1}{2}(A-J\cdot A\cdot J)\circ ds[\cdot]\right)\cdot v\right]\right)[Q]\right)^{0,1}\cdot\theta\\
&\qquad=\left(\left[s\,,\;\left(Q^J\circ ds[\cdot]\right)v\right]\right)^{0,1}\cdot\theta\\
&\qquad=\left[s\,,\;\left((Q^J\circ d\os)^{0,1}\circ d\varphi[\cdot]\right)v\right]\cdot\theta\\
&\qquad=:\left((Q^J\circ d\os)^{0,1}\circ d\varphi\right)\psi
\end{align*}
We finally calculate the differential with respect to $J$.
Let $Y\in T_J\mJ^{l+1}$.
By the product rule (using that all three terms occurring are linear),
\begin{align*}
&D_{(\varphi,\psi,A,J)}\mF_2[Y]\\
&\qquad=D_{J'=J}\left(\tL_{\varphi}^A(J,J')^{-1}\circ\mD_{\varphi}^{A,J'}
\circ L_{\varphi}^A(J,J')\circ\psi\right)[Y]\\
&\qquad=D_{J'=J}\left(\tL_{\varphi}^A(J,J')^{-1}\right)[Y]
\circ\left(\mD_{\varphi}^{A,J'}\circ L_{\varphi}^A(J,J')\circ\psi\right)|_{J'=J}\\
&\qquad\quad+\left(\tL_{\varphi}^A(J,J')^{-1}\right)|_{J'=J}\circ
D_{J'=J}\left(\mD_{\varphi}^{A,J'}\right)[Y]\circ\left(L_{\varphi}^A(J,J')\circ\psi\right)|_{J'=J}\\
&\qquad\quad+\left(\tL_{\varphi}^A(J,J')^{-1}\right)|_{J'=J}\circ
\left(\mD_{\varphi}^{A,J'}\right)|_{J'=J}\circ D_{J'=J}\left(L_{\varphi}^A(J,J')\circ\psi\right)[Y]\\
&\qquad=D_{J'=J}\left(\tL_{\varphi}^A(J,J')^{-1}\right)[Y]\circ\mD_{\varphi}^{A,J}\circ\psi
+D_{J'=J}\left(\mD_{\varphi}^{A,J'}\right)[Y]\circ\psi\\
&\qquad\qquad+\mD_{\varphi}^{A,J}\circ D_{J'=J}\left(L_{\varphi}^A(J,J')\right)[Y]\circ\psi
\end{align*}
holds. Abbreviating $L:=\tL_{\varphi}^A(J,J')$, we yield
\begin{align*}
0&=D_{J'=J}(L^{-1}\circ L)[Y]\\
&=D_{J'=J}(L^{-1})[Y]\circ L|_{J'=J}+L^{-1}|_{J'=J}\circ D_{J'=J}(L)[Y]\\
&=D_{J'=J}(L^{-1})[Y]+D_{J'=J}(L)[Y]
\end{align*}
and thus
\begin{align*}
&D_{(\varphi,\psi,A,J)}\mF_2[Y]\\
&\qquad=-D_{J'=J}\left(\tL_{\varphi}^A(J,J')\right)[Y]\circ\mD_{\varphi}^{A,J}\circ\psi
+D_{J'=J}\left(\mD_{\varphi}^{A,J'}\right)[Y]\circ\psi\\
&\qquad\qquad+\mD_{\varphi}^{A,J}\circ D_{J'=J}\left(L_{\varphi}^A(J,J')\right)[Y]\circ\psi\\
&\qquad=\frac{1}{2}JY\circ\mD_{\varphi}^{A,J}\psi+D_{J'=J}\left(\mD_{\varphi}^{A,J'}\psi\right)[Y]
-\frac{1}{2}\mD_{\varphi}^{A,J}\circ JY\psi
\end{align*}
where we used the explicit formulas for $L_{\varphi}^A$ and $\tL_{\varphi}^A$ as stated above and
that (on the complex tensor product) $i=J$ holds.
Moreover, we calculate
\begin{align*}
&D_{J'=J}\left(\mD_{\varphi}^{A,J'}\psi\right)[Y]\\
&\qquad=D_{J'=J}\left(\left(\nabla^{A,J'}\psi_{\theta}\right)^{0,1}\right)[Y]\cdot\theta\\
&\qquad=\frac{1}{4}D_{J'=J}\left(-J'\nabla^A(J'\psi_{\theta})
+J'\nabla_{j[\cdot]}^A\psi_{\theta}+\nabla_{j[\cdot]}^A(J'\psi_{\theta})\right)[Y]\cdot\theta\\
&\qquad=\frac{1}{4}\left(-Y\nabla^A(J\psi_{\theta})-J\nabla^A(Y\psi_{\theta})
+Y\nabla_{j[\cdot]}^A\psi_{\theta}+\nabla_{j[\cdot]}^A(Y\psi_{\theta})\right)\cdot\theta\\
&\qquad=:D_Y(\varphi,\psi,A,J)[Y]
\end{align*}
which concludes the proof.
\end{proof}

\begin{Prp}
\label{prpUniversalModuliSpace}
Let $\beta\in H_2(X,\bZ)$, $l\geq 2$, $p>2$ and $k\in\{1,\ldots,l\}$. Then the universal
moduli space $\hat{\mM}^*(\beta,\Sigma,L;\mA^{l+1},\mJ^{l+1})$ is a separable $C^{l-k}$-Banach submanifold
of $\mG^{k,p}$.
\end{Prp}

\begin{proof}
We first show that $D_{(\varphi,\psi,A,J)}\mF$, as calculated in
Lem. \ref{lemUniversalVerticalDifferential}, is surjective whenever
$(\varphi,\psi,A,J)\in\hat{\mM}^*(\beta,\Sigma,L;\mA^{l+1},\mJ^{l+1})$.
It is well-known that the map $(\Xi+Y)\mapsto D_{(\varphi,\psi,A,J)}\mF_1[\Xi+Y]$,
i.e. the restriction of $D_{(\varphi,\psi,A,J)}\mF_1$ to the subspace
$W^{k,p}(\Sigma,\varphi^*TX)\oplus C^{l+1}(X,\mathrm{End}(TX,J,\omega))$,
is surjective (as shown in the proof of Prp. 3.2.1 in \cite{MS04}) and, therefore,
surjectivity of the entire differential follows from surjectivity of the restriction
\begin{align*}
&\tD_{(\varphi,\psi,A,J)}\mF_2:(\Psi+Q)\\
&\qquad\qquad\mapsto D_{(\varphi,\psi,A,J)}\mF_2[\Psi+Q]
=\mD_{\varphi}^{A,J}\Psi+\left((Q^J\circ d\os)^{0,1}\circ d\varphi\right)\psi
\end{align*}
We show that $\tD_{(\varphi,\psi,A,J)}\mF_2$ is surjective.
Firstly, $\Psi\mapsto\mD_{\varphi}^{A,J}\Psi$ is a Fredholm operator by Lem. \ref{lemComplexCauchyRiemann},
and as such has a closed image and a finite dimensional cokernel. Since
$\tD_{(\varphi,\psi,A,J)}\mF_2$ is the sum of $\mD_{\varphi}^{A,J}$ with another bounded linear operator,
it follows that its image is also closed. For surjectivity, it thus suffices to prove that the image is
dense whenever $(\varphi,\psi,A,J)\in\hat{\mM}^*(\beta,\Sigma,L;\mA^{l+1},\mJ^{l+1})$.

First, consider $k=1$. If the image is not dense then, by the theorem of Hahn-Banach, there exists a nonzero section
$\eta\in L^q(\Sigma,\Lambda^{0,1}\Sigma\otimes_{\bC}L\otimes_J\varphi^*TX)$ with $1/p+1/q=1$ such that
the integral
\begin{align*}
\int_{\Sigma}\scal{\eta}{\tD_{(\varphi,\psi,A,J)}\mF_2[Q+\Psi]}\dvol_{\Sigma}=0
\end{align*}
vanishes for all $Q\in\Omega^1_{C^{l+1}}(\Sigma,GL(X)\times_{Ad}\fg)$ and
$\Psi\in W^{1,p}(\Sigma,L\otimes_J\varphi^*TX)$. In particular, this implies
\begin{align*}
\int_{\Sigma}\scal{\eta}{\mD_{\varphi}^{A,J}\Psi}\dvol_{\Sigma}=0
\end{align*}
and thus, by the Riemann-Roch theorem stated as Prp. \ref{prpRiemannRoch} and applied to $\mD_{\varphi}^{A,J}$,
$\eta\in W^{1,p}$ is of the same regularity class as $\Psi$ and, by the Sobolev embedding (\ref{eqnMorrey}),
$\eta$ is moreover continuous.
Let $z_0\in\Sigma$ be such that
\begin{align*}
d_{z_0}\varphi\neq 0\;,\qquad\varphi^{-1}(\varphi(z_0))=\{z_0\}\;,\qquad\psi(z_0)\neq 0
\end{align*}
Since, by assumption, $\varphi$ is simple and $\psi$ is not the zero section, the set of such
points is open and dense in $\Sigma$ (this follows from Prp. 2.5.1 in \cite{MS04}, and
from Lem. \ref{lemCROperatorZeroes}).
Assume, by contradiction, that $\eta(z_0)\neq 0$.
Let $V$, $U$, $\theta$ and $\os$ be, respectively, open sets and local sections as in Lem.
\ref{lemUniversalVerticalDifferential} and such that $z_0\in U$.
Then, on $U$, we can decompose $\psi=\theta\cdot\psi_{\theta}$ and $\eta=\theta\cdot\eta_{\theta}$
with $\psi_{\theta}\in\Gamma_{C^l}(U,\varphi^*TX)$ and
$\eta_{\theta}\in L^q(U,\Lambda^{0,1}\Sigma\otimes_J\varphi^*TX)$.
On $V$, we identify $J$ with a matrix-valued function in the trivialisation of $TX$ determined by $\os$.
For $g\in\fg:=gl(2n,\bR)$, we denote by
\begin{align*}
g^J:=\frac{1}{2}(g-J\cdot g\cdot J)\in C^{l+1}(V,\fg)
\end{align*}
the complex linearisation with respect to $J$. There is a canonical
identification with the element $g^J\cong\frac{1}{2}(g-J_0\cdot g\cdot J_0)\in\fg_{\bC}:=gl(n,\bC)$,
where $J_0$ denotes the standard complex structure on $\bR^{2n}$.
We let $\lambda\in\Lambda^{0,1}_{z_0}\Sigma$ and $\xi\in T_{\varphi(z_0)}X$ be such that
\begin{align*}
\eta_{\theta}(z_0)=\lambda\cdot\xi\in\Lambda^{0,1}_{z_0}\Sigma\otimes_{J(\varphi(z_0))}T_{\varphi(z_0)}X
\end{align*}
Now, since by assumption $0\neq\psi_{\theta}(z_0)\in T_{\varphi(z_0)}X$ and $d_{z_0}\varphi\neq 0$,
and since $\fg_{\bC}$ acts transitively on $\bR^{2n}$, there are $g\in\fg$ and $q\in T^*_{\varphi(z_0)}X$ such that
\begin{align*}
g^{J(\varphi(z_0))}\cdot\psi_{\theta}(z_0)=\xi\;,\qquad
(q\circ d_{z_0}\varphi)^{0,1}=\lambda
\end{align*}
We define
\begin{align}
\label{eqnQ0}
Q_0:=q^{0,1}\cdot g^{J(\varphi(z_0))}\in\Lambda^{0,1}_{\varphi(z_0)}X\otimes_{J(\varphi(z_0))}\fg_{\bC}
\end{align}
and obtain, using that $\varphi$ satisfies $J\circ d\varphi=d\varphi\circ j$,
\begin{align*}
(Q_0\circ d_{z_0}\varphi)\cdot\psi(z_0)
&=(q^{0,1}\circ d_{z_0}\varphi)\cdot\left(g^{J(\varphi(z_0))}\cdot\psi_{\theta}(z_0)\right)\cdot\theta(z_0)\\
&=(q\circ d_{z_0}\varphi)^{0,1}\cdot\xi\cdot\theta(z_0)\\
&=\lambda\cdot\xi\cdot\theta(z_0)=\eta(z_0)
\end{align*}
and, in particular, $\scal{\eta(z_0)}{(Q_0\circ d_{z_0}\varphi)\cdot\psi(z_0)}>0$.
Now let $Q\in\mQ:=\Omega^1_{C^{l+1}}(\Sigma,GL(X)\times_{Ad}\fg)$.
Then $Q\circ d\os\in\Omega^1_{C^{l+1}}(V,\fg)$ and, for every $p\in V$ and every
$\omega\in\Lambda^1_pX\otimes_{\bR}\fg$, there is $Q\in\mQ$ such that
$(Q\circ d\os)_p=\omega$ (cf. \cite{Bau09}, Sec. 3.2).
A fortiori, for every $\omega\in\Lambda^1_pX\otimes_{\bR}\fg_{\bC}$, there is $Q\in\mQ$ such that
$(Q^J\circ d\os)_p=\omega$ upon identifying $(\fg_{\bC},J_0)$ with the image under the
projection $g\mapsto g^{J(p)}$ with complex structure $J(p)$ as above.
Projecting again, we find that, for every $\omega\in\Lambda^{0,1}_pX\otimes_{J(p)}\fg_{\bC}$,
there is $Q\in\mQ$ such that $(Q^J\circ d\os)^{0,1}_p=\omega$.
Now let $Q\in\mQ$ be such that $(Q^J\circ d\os)^{0,1}_{\varphi(z_0)}=Q_0$ with $Q_0$ as in (\ref{eqnQ0}). Then
\begin{align*}
(\tD_{(\varphi,\psi,A,J)}\mF_2[Q])_{z_0}
=(Q_0\circ d_{z_0}\varphi)\psi(z_0)=\eta(z_0)
\end{align*}
and, by continuity, the function $\scal{\eta}{\tD_{(\varphi,\psi,A,J)}\mF_2[Q]}$ on $\Sigma$ is
positive in a neighbourhood $U_0$ of $z_0$. Because $z_0$ is such that
$\varphi^{-1}(\varphi(z_0))=\{z_0\}$, this implies that $\varphi(z_0)\notin\varphi(\Sigma\setminus U_0)$.
Therefore, there is a neighbourhood $\varphi(z_0)\ni V_0\subseteq X$ such that
$\varphi(\Sigma\setminus U_0)\cap V_0=\emptyset$ and hence $\varphi^{-1}(V_0)\subseteq U_0$.
We choose a cutoff function $\beta:X\rightarrow[0,1]$ with support in $V_0$ and such that
$\beta(\varphi(z_0))=1$. Then the function
\begin{align*}
(\beta\circ\varphi)\cdot\scal{\eta}{\tD_{(\varphi,\psi,A,J)}\mF_2[Q]}=
\scal{\eta}{\tD_{(\varphi,\psi,A,J)}\mF_2[(\beta\circ\varphi)\cdot Q]}
\end{align*}
on $\Sigma$ is supported in $U_0$, is nonnegative and somewhere positive, whence
\begin{align*}
\int_{\Sigma}\scal{\eta}{\tD_{(\varphi,\psi,A,J)}\mF_2[(\beta\circ\varphi)\cdot Q]}\dvol_{\Sigma}\neq 0
\end{align*}
This contradiction shows that $\eta(z_0)=0$. As the set of $z_0$ for which this reasoning holds
is dense in $\Sigma$, $\eta$ vanishes almost everywhere and, by continuity, $\eta\equiv 0$
follows. This, in turn, is a contradiction to the image of $\tD_{(\varphi,\psi,A,J)}\mF_2$
not being dense and, therefore,
we have shown that the image is dense and (since it is also closed) surjective in the case $k=1$.

To prove surjectivity for general $k$, let
$\eta\in W^{k-1,p}(\Sigma,\Lambda^{0,1}\Sigma\otimes_{\bC}L\otimes_J\varphi^*TX)$. Now we use surjectivity
for $k=1$: In fact we have just proved that there exist $\Psi$ and $Q$ of regularity class $W^{1,p}$ and
$C^{l+1}$, respectively, such that $\tD_{(\varphi,\psi,A,J)}\mF_2[\Psi+Q]=\eta$. Therefore,
\begin{align*}
\mD^{A,J}_{\varphi}\Psi=\eta-\left((Q^J\circ d\os)^{0,1}\circ d\varphi\right)\psi\in W^{k-1,p}
\end{align*}
By the local form of the operator $\mD^{A,J}_{\varphi}$ (cf. Lem. \ref{lemLocalHolomorphicSupercurve}),
we may thus apply the elliptic bootstrapping method from Lem. \ref{lemEllipticBootstrapping} to conclude
that $\Psi\in W^{k,p}$ holds. This concludes the proof that $\tD_{(\varphi,\psi,A,J)}\mF_2$ and,
therefore, also $D_{(\varphi,\psi,A,J)}\mF$ is surjective for general $k\leq l$ whenever
$(\varphi,\psi,A,J)\in\hat{\mM}^*(\beta,\Sigma,L;\mA^{l+1},\mJ^{l+1})$.

To finish the proof, note that surjectivity of the map
\begin{align*}
(\Xi+Y)\mapsto D_{(\varphi,\psi,A,J)}\mF_1[\Xi+Y]=D_{\varphi}\Xi+\frac{1}{2}Y(\varphi)\circ d\varphi\circ j
\end{align*}
implies that it has a right inverse $R$ since $D_{\varphi}$ is a Fredholm operator. By the same reasoning,
$\tD_{(\varphi,\psi,A,J)}\mF_2$ has a right inverse $\hat{R}$,
and it follows that $R+\hat{R}$ is a right inverse of $D_{(\varphi,\psi,A,J)}\mF$.
Hence it follows from the infinite dimensional implicit function theorem
that $\hat{\mM}^*(A,\Sigma,L;\mA^{l+1},\mJ^{l+1})$ is a $C^{l-k}$-Banach
submanifold of $\mG^{k,p}$, which is separable since $\mG^{k,p}$ is.
\end{proof}

\subsection{Proof of Transversality}
\label{subsecProofsTransversality}

Having established the necessary background by now, we come to
the proof of Thm. \ref{thmSuperregular}. We need two auxiliary lemmas.

\begin{Lem}
\label{lemGEBanachBundle}
Let $1\leq k\leq l$ be natural numbers, $p>2$, $\beta\in H_2(X,\bZ)$ and $(A,J)\in\mA^{l+1}\times\mJ^{l+1}$. Then, prescribing
\begin{align*}
\fG^{k,p}_{\varphi}&:=W^{k,p}(\Sigma,L\otimes_J\varphi^*TX)\\
\fE^{k-1,p}_{\varphi}&:=W^{k-1,p}(\Sigma,\Lambda^{0,1}\Sigma\otimes_J\varphi^*TX)\\
\hat{\fE}^{k-1,p}_{\varphi}&:=W^{k-1,p}(\Sigma,\Lambda^{0,1}\Sigma\otimes_{\bC}L\otimes_J\varphi^*TX)
\end{align*}
for $\varphi\in\mB^{k,p}_{\beta}$ defines separable $C^{l-k}$-Banach space bundles
\begin{align*}
\fG^{k,p},\fE^{k-1,p},\hat{\fE}^{k-1,p}\rightarrow\mB^{k,p}_{\beta}
\end{align*}
\end{Lem}

\begin{proof}
This is shown completely analogous to Lem. \ref{lemGBanachBundle} and Lem \ref{lemEBanachBundle}.
Here, we only need trivialisation via parallel transport.
\end{proof}

Similar to the discussion preceeding (\ref{eqnUniversalModuliSpace}), elliptic regularity implies that
we may identify the moduli space of regularity class $C^l$ with a subspace of $\fG^{k,p}$ as follows.
For $\psi\in\fG^{k,p}$, we abbreviate $\varphi:=\pi_{\mB}(\psi)$. Then
\begin{align*}
\hat{\mM}^*(\beta,\Sigma,L;A,J)
=\{\psi\in\fG^{k,p}\setminus\{0\}\setsep\dbar_J\varphi=0\,,\;\mD_{\varphi}^{A,J}\psi=0\}
\end{align*}
Moreover, defining a map $\fF:\fG^{k,p}\setminus\{0\}\rightarrow\fE^{k-1,p}\oplus\hat{\fE}^{k-1,p}$ by
\begin{align}
\label{eqnFF}
\fF(\psi):=\fF_1(\psi)+\fF_2(\psi):=\dbar_J\varphi+\mD_{\varphi}^{A,J}\psi
\end{align}
it is clear that
\begin{align}
\hat{\mM}^*(\beta,\Sigma,L;A,J)=\fF^{-1}(0)
\end{align}

\begin{Lem}
\label{lemVerticalDifferential}
Let $l\geq 2$, $p>2$ and $k\in\{1,\ldots,l\}$. Then $\fF$ is a function of regularity class $C^{l-k}$.
Moreover, its vertical differential is the map
\begin{align*}
&D_{(\varphi,\psi)}(\fF_1+\fF_2)=\hat{\mD}_{\varphi,\psi}^{A,J}:
W^{k,p}(\Sigma,\varphi^*TX)\oplus W^{k,p}(\Sigma,L\otimes_J\varphi^*TX)\\
&\qquad\rightarrow W^{k-1,p}(\Sigma,\Lambda^{0,1}\Sigma\otimes_J\varphi^*TX)\oplus
W^{k-1,p}(\Sigma,\Lambda^{0,1}\Sigma\otimes_{\bC}L\otimes_J\varphi^*TX)
\end{align*}
defined by (\ref{eqnVerticalDifferential}).
\end{Lem}

\begin{proof}
This follows directly from Lem. \ref{lemUniversalVerticalDifferential} upon setting $Q=0$ and $Y=0$.
\end{proof}

\begin{proof}[Proof of Thm. \ref{thmSuperregular}(i)]
Let $(A,J)\in\mA\mJ_{\mathrm{reg}}(\beta,\Sigma,L)$ be regular as assumed in the hypotheses, and fix
$(\varphi,\psi)\in\hat{\mM}^*(\beta,\Sigma,L;A,J)$.
Let $k\geq 1$ be an integer and $p>2$.
By Lem. \ref{lemVerticalDifferential}, $\fF$ is locally a smooth map between Banach spaces.
More precisely, restricted to a sufficiently small neighbourhood $U(\varphi)\subseteq\mB_{\beta}^{k,p}$
of $\varphi$, the Banach space bundles involved are trivial (by Lem. \ref{lemGEBanachBundle}),
and $\fF$ can be considered as a map
\begin{align*}
&\fF:U(\varphi)\times W^{k,p}(\Sigma,L\otimes_J\varphi^*TX)\setminus\{0\}\\
&\qquad\rightarrow W^{k-1,p}(\Sigma,\Lambda^{0,1}\Sigma\otimes_J\varphi^*TX)\oplus
W^{k-1,p}(\Sigma,\Lambda^{0,1}\Sigma\otimes_{\bC}L\otimes_J\varphi^*TX)
\end{align*}
where $U(\varphi)$ is identified with a neighbourhood of zero in $W^{k,p}(\Sigma,\varphi^*TX)$
like in the proof of Lem. \ref{lemGBanachBundle}.
By Lem. \ref{lemVerticalDifferential} and Lem. \ref{lemCROperatorFredholmSum}, we know that
the differential $D_{(\varphi,\psi)}\fF$ is Fredholm with index $\mathrm{ind}D_{\varphi}+\mathrm{ind}\mD_{\varphi}^{A,J}$.
Since $(A,J)$ is regular, it follows from the discussion below Prp. \ref{prpRiemannRoch}
that $D_{(\varphi,\psi)}\fF$ is, moreover, surjective and, being Fredholm, also has a right inverse.
The hypotheses of the infinite dimensional implicit function theorem
are thus satisfied, and it follows that $\fF^{-1}(0)$ intersects the local domain stated
in a smooth submanifold of dimension $\mathrm{ind}D_{\varphi}+\mathrm{ind}\mD_{\varphi}^{A,J}$.
Hence $\hat{\mM}^*(\beta,\Sigma,L;A,J)$ is a smooth submanifold of $\fG^{k,p}$.
According to the discussion following Lem. \ref{lemGEBanachBundle},
the coordinate charts of $\hat{\mM}^*(\beta,\Sigma,L;A,J)$ obtained in this way
are (by elliptic regularity) independent of the choices of $k$ and $p$.

It remains to show that the moduli spaces have a natural orientation.
This is a standard argument (carried out e.g. in the proof of Prp. 2.8 in \cite{Sol06}).
As shown in App. A.2 of \cite{MS04},
prescribing $\det D:=\bigwedge(\ker D)\otimes\bigwedge(\ker D^*)$ for a
Fredholm operator $D:X\rightarrow Y$ between two Banach spaces $X$ and $Y$
yields a line bundle over the space of Fredholm operators.
If $D$ is complex linear, $\ker D$ and $\ker D^*$ carry an almost complex structure,
thus inducing an orientation on $\det D$. The last conclusion remains true if $D$
is a real-linear Cauchy-Riemann operator: In this case, trivialising the determinant
line bundle over any homotopy from $D$ to its complex-linear part yields an orientation
on $\det D$.
Applied to our situation, $\hat{\mD}^{A,J}_{\varphi,\psi}$ is a surjective real
linear Cauchy-Riemann operator and, therefore, we obtain a canonical orientation
on $\det\hat{\mD}^{A,J}_{\varphi,\psi}=\bigwedge(\ker\hat{\mD}^{A,J}_{\varphi,\psi})$.
Since, again by the implicit function theorem, its kernel coincides with the tangent space
$T_{(\varphi,\psi)}\hat{\mM}^*(\beta,\Sigma,L;A,J)$, and this yields an orientation of
the moduli space $\hat{\mM}^*(\beta,\Sigma,L;A,J)$.
\end{proof}

\begin{proof}[Proof of Thm. \ref{thmSuperregular}(ii)]
Let $p>2$. We show that the set
\begin{align*}
\mA\mJ^l_{\mathrm{reg}}&:=\mA\mJ^l_{\mathrm{reg}}(\beta,\Sigma,L)\\
&:=\{(A,J)\in\mA^l\times\mJ^l\setsep\hat{\mD}_{\varphi,\psi}^{A,J}\;\textrm{surj. for all}\;
(\varphi,\psi)\in\hat{\mM}^*(\beta,\Sigma,L;A,J)\}
\end{align*}
is dense in $\mA^l\times\mJ^l$ for $l$ sufficiently large and then use a Taubes argument to prove the theorem.
By Prp. \ref{prpUniversalModuliSpace}, applied with $k=1$, the projection
\begin{align*}
\pi:\hat{\mM}^*(\beta,\Sigma,L;\mA^{l+1},\mJ^{l+1})\rightarrow\mA^{l+1}\times\mJ^{l+1}
\end{align*}
is $C^{l-1}$-map between separable $C^{l-1}$-Banach manifolds. From (\ref{eqnUniversalModuliSpace}),
it is clear that the tangent space of the domain equals
\begin{align*}
&T_{(\varphi,\psi,A,J)}\hat{\mM}^*(\beta,\Sigma,L;\mA^{l+1},\mJ^{l+1})\\
&\qquad\qquad=\{(\Xi+\Psi+Q+Y)\setsep D_{(\varphi,\psi,A,J)}\mF[\Xi,\Psi,Q,Y]=0\}
\end{align*}
where the tuple $(\Xi,\Psi,Q,Y)$ is as in Lem. \ref{lemUniversalVerticalDifferential} and,
by that lemma, the tangent condition can be rewritten as follows.
\begin{align*}
0=&\hat{\mD}_{\varphi,\psi}^{A,J}[\Xi+\Psi]
+\frac{1}{2}Y(\varphi)\circ d\varphi\circ j
+\left((Q^J\circ d\os)^{0,1}\circ d\varphi\right)\psi\\
&+\frac{1}{2}JY\circ\mD_{\varphi}^{A,J}\psi
-\frac{1}{2}\mD_{\varphi}^{A,J}\circ JY\psi+D_Y(\varphi,\psi,A,J)[Y]
\end{align*}
Moreover, the derivative of $\pi$ is the projection
\begin{align*}
d_{(\varphi,\psi,A,J)}\pi:T_{(\varphi,\psi,A,J)}\hat{\mM}^*(\beta,\Sigma,L;\mA^{l+1},\mJ^{l+1})
&\rightarrow T_A\mA^{l+1}\oplus T_J\mJ^{l+1}\;,\\
(\Xi+Q+Y+\Psi)&\mapsto Q+Y
\end{align*}
whence the kernel of $d_{(\varphi,\psi,A,J)}\pi$ is isomorphic to the kernel of
$\hat{\mD}_{\varphi,\psi}^{A,J}$. It follows that the cokernels are also isomorphic
(cf. \cite{MS04}, Lem. A.3.6). Therefore, $d_{(\varphi,\psi,A,J)}\pi$ is Fredholm
with the same index as $\hat{\mD}_{\varphi,\psi}^{A,J}$ and, moreover,
$d_{(\varphi,\psi,A,J)}\pi$ is surjective if and only if $\hat{\mD}_{\varphi,\psi}^{A,J}$ is.
A regular value $(A,J)$ of $\pi$ is thus characterised by the property that
$\hat{\mD}_{\varphi,\psi}^{A,J}$ is surjective for every
$(\varphi,\psi)\in\hat{\mM}^*(\beta,\Sigma,L;A,J)=\pi^{-1}((A,J))$.
The set of regular values of $\pi$ is hence precisely $\mA\mJ^{l+1}_{\mathrm{reg}}$.
By the theorem of Sard-Smale (cf. \cite{Sma65}),
this set is of the second category in $\mA^{l+1}\times\mJ^{l+1}$, provided that
\begin{align*}
l-2\geq\mathrm{ind}\pi=\mathrm{ind}\hat{\mD}_{\varphi,\psi}^{A,J}=n(4-4g)+4c_1(\beta)+2n\deg L
\end{align*}
Therefore, $\mA\mJ_{\mathrm{reg}}^l$ is dense in $\mA^l\times\mJ^l$ with respect to the $C^l$-topology
for sufficiently large $l$.

We shall now use an argument due to Taubes, to show that $\mA\mJ_{\mathrm{reg}}$ is of the
second category in $\mA\times\mJ$ with respect to the $C^{\infty}$-topology,
which largely parallels the one in the proof of Thm. 3.1.5(ii) in \cite{MS04}. Let
\begin{align*}
\mA\mJ_{\mathrm{reg},K}\subseteq\mA\times\mJ
\end{align*}
be the set of all smooth $(A,J)\in\mA\times\mJ$ such that $\hat{\mD}_{\varphi,\psi}^{A,J}$ is surjective
for every $(\varphi,\psi)\in\hat{\mM}^*(\beta,\Sigma,L;A,J)$ for which there exists $z\in\Sigma$ such that
\begin{align}
\label{eqnTaubes}
\norm{d\varphi}_{L^{\infty}}\leq K\;,\qquad
\inf_{\zeta\neq z}\frac{d(\varphi(z),\varphi(\zeta))}{d(z,\zeta)}\geq\frac{1}{K}\;,\qquad
\frac{1}{K}\leq\norm{\psi}_{L^{\infty}}\leq K
\end{align}
Every $(\varphi,\psi)\in\hat{\mM}^*(\beta,\Sigma,L;A,J)$ satisfies these conditions for some value of
$K>0$ and some point $z\in\Sigma$. Hence
\begin{align*}
\mA\mJ_{\mathrm{reg}}=\bigcap_{K>0}\mA\mJ_{\mathrm{reg},K}
\end{align*}
We prove that each $\mA\mJ_{\mathrm{reg},K}$ is open in $\mA\times\mJ$ or, equivalently,
that its complement is closed. Hence assume that a sequence
$(A^{\nu},J^{\nu})\notin\mA\mJ_{\mathrm{reg},K}$ converges to $(A,J)\in\mA\times\mJ$ in the
$C^{\infty}$-topology. Then there exist, for every $\nu$, a point $z^{\nu}\in\Sigma$
and a pair $(\varphi^{\nu},\psi^{\nu})\in\hat{\mM}^*(\beta,\Sigma,L;A^{\nu},J^{\nu})$
that satisfies (\ref{eqnTaubes}) with $z$ replaced by $z^{\nu}$, such that
$\hat{\mD}_{\varphi^{\nu},\psi^{\nu}}^{A^{\nu},J^{\nu}}$ is not surjective. By the
compactness result from Prp. \ref{prpPhiPsiConverging}, it follows that
there exists a subsequence $(\varphi^{\nu_i},\psi^{\nu_i})$ which converges,
uniformly with all derivatives, to a smooth pair $(\varphi,\psi)\in\hat{\mM}^*(\beta,\Sigma,L;A,J)$.
Since $\Sigma$ is compact, we may choose the subsequence such that $z^{\nu_i}$ converges to $z\in\Sigma$. Then
$(\varphi,\psi)$ satisfies (\ref{eqnTaubes}) for this point $z$ and moreover,
since the operators $\hat{\mD}_{\varphi^{\nu},\psi^{\nu}}^{A^{\nu},J^{\nu}}$ are not surjective, it
follows that $\hat{\mD}_{\varphi,\psi}^{A,J}$ is not surjective either.
This shows that $(A,J)\notin\mA\mJ_{\mathrm{reg},K}$, and thus we have proved that
the complement of $\mA\mJ_{\mathrm{reg},K}$ is closed in the $C^{\infty}$-topology.
Let $\mA\mJ_{\mathrm{reg},K}^l$ denote the set of all $(A,J)\in\mA^l\times\mJ^l$
such that $\hat{\mD}_{\varphi,\psi}^{A,J}$ is surjective for every
$(\varphi,\psi)\in\hat{\mM}^*(\beta,\Sigma,L;A,J)$
of regularity class $C^{l-1}$ that satisfies (\ref{eqnTaubes}) for some $z\in\Sigma$.
By the same argument as just provided,
$\mA\mJ_{\mathrm{reg},K}^l$ is open in $\mA^l\times\mJ^l$ with respect to the $C^l$-topology. Moreover,
note that
\begin{align*}
\mA\mJ_{\mathrm{reg},K}=\mA\mJ_{\mathrm{reg},K}^l\cap\mA\times\mJ
\end{align*}
We show next that $\mA\mJ_{\mathrm{reg},K}$ is dense in $\mA\times\mJ$ with respect to the $C^{\infty}$-topology.
Let $(A,J)\in\mA\times\mJ$. Since $\mA\mJ_{\mathrm{reg}}^l$ is dense in $\mA^l\times\mJ^l$ for large $l$,
there exists a sequence $(A^l,J^l)\in\mA\mJ_{\mathrm{reg}}^l$, $l\geq l_0$, such that
\begin{align*}
\norm{(A,J)-(A^l,J^l)}_{C^l}\leq 2^{-l}
\end{align*}
Since $(A^l,J^l)\in\mA\mJ_{\mathrm{reg},K}^l$ and $\mA\mJ_{\mathrm{reg},K}^l$ is open in the $C^l$-topology
there exists an $\varepsilon^l>0$ such that, for every $(A',J')\in\mA^l\times\mJ^l$, the implication
\begin{align*}
\norm{(A',J')-(A^l,J^l)}_{C^l}<\varepsilon^l\quad\implies\quad
(A',J')\in\mA\mJ_{\mathrm{reg},K}^l
\end{align*}
is true. Choose $(A'^l,J'^l)$ to be any smooth element such that
\begin{align*}
\norm{(A'^l,J'^l)-(A^l,J^l)}_{C^l}\leq\min\{\varepsilon^l,2^{-l}\}
\end{align*}
Then
\begin{align*}
(A'^l,J'^l)\in\mA\mJ_{\mathrm{reg},K}^l\cap\mA\times\mJ=\mA\mJ_{\mathrm{reg},K}
\end{align*}
and the sequence $(A'^l,J'^l)$ converges to $(A,J)$ in the $C^{\infty}$-topology.
This shows that the set $\mA\mJ_{\mathrm{reg},K}$ is dense in $\mA\times\mJ$ as claimed.
Thus $\mA\mJ_{\mathrm{reg}}$ is the intersection of the countable number of open dense sets
$\mA\mJ_{\mathrm{reg},K}$, $K\in\bN$, and so is of the second category as required.
\end{proof}

\subsection{Dependence on $(A,J)$}

We close this article by establishing the following result, to be made precise in
Thm. \ref{thmSuperregularHomotopy} below:
For two regular pairs $(A_0,J_0),(A_1,J_1)\in\mA\mJ_{\mathrm{reg}}(\beta,\Sigma,L)$,
the manifolds $\hat{\mM}^*(\beta,\Sigma,L;A_0,J_0)$ and
$\hat{\mM}^*(\beta,\Sigma,L;A_1,J_1)$ are oriented cobordant
(by using this term, we do not mean to imply that the spaces in question are compact).
The proof will turn out to be similar to that of Thm. \ref{thmSuperregular}.

\begin{Def}
A \emph{(smooth) homotopy} of connections and almost complex structures
is a smooth map $H:[0,1]\rightarrow\mA\times\mJ\,,\;H(\lambda)=(A_{\lambda},J_{\lambda})$.
Given such a homotopy, we define
\begin{align*}
\hat{\mW}^*(\beta,\Sigma,L;H):=
\{(\lambda,\varphi,\psi)\setsep 0\leq\lambda\leq 1\,,\;
(\varphi,\psi)\in\hat{\mM}^*(\beta,\Sigma,L;A_{\lambda},J_{\lambda})\}
\end{align*}
\end{Def}

The definition of $\hat{\mW}^*$ with $\mA\times\mJ$ replaced by $\mA^{l+1}\times\mJ^{l+1}$
is analogous.
Given a homotopy $H$ and $\lambda\in[0,1]$, we denote by $\fF_{\lambda}$ the map (\ref{eqnFF})
with $(A,J)$ replaced by $(A_{\lambda},J_{\lambda})\in\mA^{l+1}\times\mJ^{l+1}$.
Moreover, introducing the map
\begin{align*}
\tilde{\fF}:[0,1]\times\fG^{k,p}\setminus\{0\}&\rightarrow\fE^{k-1,p}\oplus\hat{\fE}^{k-1,p}\\
\tilde{\fF}(\lambda,\psi)&:=\fF_{\lambda}(\psi)=({\fF}_1)_{\lambda}(\psi)+({\fF}_1)_{\lambda}(\psi)
\end{align*}
it is clear that
\begin{align*}
\hat{\mW}^*(\beta,\Sigma,L;H)=\tilde{\fF}^{-1}(0)
\end{align*}
holds. For the rest of this section, we shall blur the distinction between spaces of
smooth maps and those of lower regularity with the implicit understanding that, exactly
as in the above treatment, we work within the categories of Sobolev and $C^l$ maps
in the first place and only at the end of the discussion pass to the smooth category,
using elliptic regularity (Prp. \ref{prpEllipticRegularity}).

\begin{Def}
Let $(A_0,J_0),(A_1,J_1)\in\mA\mJ_{\mathrm{reg}}(\beta,\Sigma,L)$. A homotopy
$H$ from $H(0)=(A_0,J_0)$ to $H(1)=(A_1,J_1)$ is called \emph{regular} if the vertical differential
$D_{(\lambda,\varphi,\psi)}\tilde{\fF}$ is surjective for every
$(\lambda,\varphi,\psi)\in\hat{\mW}^*(\beta,\Sigma,L;H)$.
We let $\mA\mJ_{\mathrm{reg}}(\beta,\Sigma,L;(A_0,J_0),(A_1,J_1))$ denote the space of regular homotopies.
\end{Def}

\begin{Thm}
\label{thmSuperregularHomotopy}
Let $(A_0,J_0),(A_1,J_1)\in\mA\mJ_{\mathrm{reg}}(\beta,\Sigma,L)$ be regular.
\begin{enumerate}
\renewcommand{\labelenumi}{(\roman{enumi})}
\item If $H=\{(A_{\lambda},J_{\lambda})\}_{\lambda}\in\mA\mJ_{\mathrm{reg}}(\beta,\Sigma,L;(A_0,J_0),(A_1,J_1))$
is a regular homotopy, then the space
$\hat{\mW}^*(\beta,\Sigma,L;H)$ is a smooth oriented manifold with
boundary
\begin{align*}
\partial\hat{\mW}^*(\beta,\Sigma,L;H)
=\hat{\mM}^*(\beta,\Sigma,L;A_0,J_0)\cup\hat{\mM}^*(\beta,\Sigma,L;A_1,J_1)
\end{align*}
whose orientation agrees with the orientation of $\hat{\mM}^*(\beta,\Sigma,L;A_1,J_1)$
and is opposite to the orientation of $\hat{\mM}^*(\beta,\Sigma,L;A_0,J_0)$.

\item The set $\mA\mJ_{\mathrm{reg}}(\beta,\Sigma,L;(A_0,J_0),(A_1,J_1))$
is of the second category in the space of all homotopies $H$
from $(A_0,J_0)$ to $(A_1,J_1)$.
\end{enumerate}
\end{Thm}

\begin{proof}[Proof of Thm. \ref{thmSuperregularHomotopy}(i)]
In a neighbourhood of $(\lambda,\varphi,\psi)\in\hat{\mW}^*$, we identify $\tilde{\fF}$ with
a map between the trivialisations of the Banach space bundles involved,
given by Lem. \ref{lemGEBanachBundle}, and analogous for the map
$\mF$ defined in (\ref{eqnMF}). Then $\tilde{\fF}=\mF\circ(\id\times H)$ is the concatenation
of $\mF$ with the homotopy $H$, and the (vertical) differential is readily calculated
to be
\begin{align*}
&D_{(\lambda,\varphi,\psi)}\tilde{\fF}[\Xi+\Psi+\Lambda]\\
&\qquad=D_{(\varphi,\psi,A_{\lambda},J_{\lambda})}\mF[\Xi+\Psi+(\partial_{\lambda}A_{\lambda})[\Lambda]
+(\partial_{\lambda}J_{\lambda})[\Lambda]]\\
&\qquad=D_{(\varphi,\psi,A_{\lambda},J_{\lambda})}\mF[\Xi+\Psi]
+D_{(\varphi,\psi,A_{\lambda},J_{\lambda})}\mF[(\partial_{\lambda}A_{\lambda})[\Lambda]
+(\partial_{\lambda}J_{\lambda})[\Lambda]]\\
&\qquad=\hat{\mD}^{A_{\lambda},J_{\lambda}}_{\varphi,\psi}[\Xi+\Psi]
+D_{(\varphi,\psi,A_{\lambda},J_{\lambda})}\mF[(\partial_{\lambda}A_{\lambda})[\Lambda]
+(\partial_{\lambda}J_{\lambda})[\Lambda]]
\end{align*}
Here the domain of $\hat{\mD}^{A_{\lambda},J_{\lambda}}_{\varphi,\psi}$ is,
compared to the original definition, endowed
by a real line corresponding to the tangent vector $\Lambda\in T_{\lambda}[0,1]\cong\bR$.
It is known to be Fredholm, while
the second summand is readily seen to be $\bR$-linear with respect to $\Lambda$
by the explicit expression given by Lem. \ref{lemUniversalVerticalDifferential}
and thus corresponds to a compact operator. It follows that
$D_{(\lambda,\varphi,\psi)}\tilde{\fF}$ is Fredholm and, by the regularity assumption,
surjective. Analogous as in the proof of Thm. \ref{thmSuperregular}(i), we may
conclude that $\hat{\mW}^*$ is a finite dimensional smooth manifold with boundary.
It is clear by definition that this boundary agrees with
$\hat{\mM}^*(\beta,\Sigma,L;A_0,J_0)\cup\hat{\mM}^*(\beta,\Sigma,L;A_1,J_1)$ as claimed.

Moreover, homotoping $D_{(\lambda,\varphi,\psi)}\tilde{\fF}$ through Fredholm operators
to the complex linear part of the (endowed) operator $\hat{\mD}^{A_{\lambda},J_{\lambda}}_{\varphi,\psi}$
induces a natural orientation on $\hat{\mW}^*$. Comparing the restriction
of this orientation to the boundary with the orientations of the moduli spaces $\hat{\mM}^*$
with respect to $(A_0,J_0)$ and $(A_1,J_1)$ as obtained in the proof of Thm. \ref{thmSuperregular}(i)
then finishes the proof.
\end{proof}

\begin{proof}[Proof of Thm. \ref{thmSuperregularHomotopy}(ii)]
We denote the space of homotopies from $(A_0,J_0)$ to $(A_1,J_1)$ by $\mH$ and define the
universal moduli space by
\begin{align*}
\hat{\mW}^*(\beta,\Sigma,L;\mH):=\{(\lambda,\varphi,\psi,H)\setsep H\in\mH\,,\;
(\lambda,\varphi,\psi)\in\hat{\mW}^*(\beta,\Sigma,L;H)\}
\end{align*}
Analogous to the proof of Thm. \ref{thmSuperregular}(ii), we have to show that
it is a separable Banach manifold, which can be seen as follows. With respect to
trivialisations of suitable Banach space bundles, consider the map
$(\lambda,\varphi,\psi,H)\mapsto\mF\circ(\id\times H(\lambda))$, whose zero set is
(an open subset of) $\hat{\mW}^*(\beta,\Sigma,L;\mH)$. By construction, the (vertical)
differential of this map is the sum of the operator $\Lambda\mapsto D_{(\lambda,\varphi,\psi)}\tilde{\fF}[\Lambda]$
(with respect to $H$) and the operator $D_{(\varphi,\psi,A_{\lambda},J_{\lambda})}\mF$
from Lem. \ref{lemUniversalVerticalDifferential}, with $(Q,Y)$ replaced by the corresponding
homotopy. It is surjective by a proof almost verbatim to that in Prp. \ref{prpUniversalModuliSpace}
and, accordingly, it follows that the universal moduli space is a separable Banach manifold.

Continuing as in the proof of Thm. \ref{thmSuperregular}(ii), one can show that the projection
onto the space $\mH$ is of sufficient regularity, and the kernel and cokernel of its differential
at $(\lambda,\varphi,\psi,H)$ are isomorphic to those of $D_{(\lambda,\varphi,\psi)}\tilde{\fF}$.
It remains to apply the theorem of Sard-Smale, followed by an argument due to Taubes as in the
second half of the proof of Thm. \ref{thmSuperregular}(ii).
We leave the details to the reader.
\end{proof}

\bibliographystyle{plain}

\end{document}